\documentclass[11pt]{amsart}

\usepackage[T1]{fontenc}
\usepackage{lmodern}
\usepackage{amsmath,amssymb,amsfonts,mathrsfs}
\usepackage{mathtools}
\usepackage{enumitem}
\usepackage{graphicx}
\usepackage{float}
\usepackage{algorithm}
\usepackage{algpseudocode}
\usepackage{microtype}
\usepackage{caption}
\usepackage{subcaption}
\usepackage{comment}
\usepackage[colorlinks=true,citecolor=blue,linkcolor=blue,urlcolor=blue]{hyperref}

\let\OriginalIncludeGraphics\includegraphics
\renewcommand{\includegraphics}[2][]{%
  \IfFileExists{#2}{\OriginalIncludeGraphics[#1]{#2}}{%
    \fbox{\parbox[c][0.28\textwidth][c]{0.86\textwidth}{\centering\small Missing figure file: \nolinkurl{#2}}}}%
}

\numberwithin{equation}{section}

\newtheorem{theorem}{Theorem}
\newtheorem{proposition}{Proposition}
\newtheorem{lemma}{Lemma}
\newtheorem{corollary}{Corollary}

\newtheorem{remark}{Remark}
\newtheorem{example}{Example}

\newcommand{\Rrig}{R}

\newcommand{\Lift}{E}

\begin{document}

\title[Minmax nonlinear elasticity]{Minmax neural-network architectures for data-to-solution value maps in nonlinear elasticity with generalized loads and variable Dirichlet data}

\author{Michael Ortiz}

\address{Division of Engineering and Applied Science, California Institute of Technology, Pasadena, CA 91125, USA.\\
Centre Internacional de M\`etodes Num\`erics en Enginyeria (CIMNE), Universitat Polit\`ecnica de Catalunya, Jordi Girona 1, 08034 Barcelona, Spain.}

\email{ortiz@caltech.edu}

\begin{abstract}
We study the data-to-solution value map for quasistatic nonlinear elasticity in the linearized-kinematics regime, allowing both generalized loads and variable Dirichlet data. Under standard direct-method hypotheses, the negative minimum potential energy is finite and locally Lipschitz, convex in the load variable, and concave in the Dirichlet datum. Its supporting slopes, and its first variations at differentiability points, are the equilibrium displacement and the Dirichlet reaction. This convex--concave structure leads to a mechanics-preserving saddle minmax architecture in which displacement atoms generate load slopes, reaction atoms generate Dirichlet slopes, and the coupling coefficients are the corresponding trace-reaction pairings. Manufactured samples are produced by prescribing displacement--reaction pairs and computing the associated ambient data and exact value labels. The resulting architecture directly approximates the negative minimum-potential-energy value map and provides mechanical subgradient readouts. Immersed representations and cell-center quadrature make the construction implementable on background grids and geometry-rich domains. We prove uniform convergence on compact data sets with respect to atom enrichment and quadrature refinement, and illustrate the method on elementary examples.
\end{abstract}

\subjclass[2020]{74B05, 74G65, 49J45, 41A30, 68T07, 65N99}
\keywords{linearized elasticity, generalized loads, Hookean energy, negative minimum potential energy, minmax neural networks, manufactured solutions, data-to-solution maps, CAD-based domains}

\maketitle

\section{Introduction}

Mesh-free and immersed discretizations were developed to reduce the dependence of numerical simulation on a body-fitted conforming mesh. Moving least squares, reproducing kernels, partition-of-unity and extended finite elements, radial-basis and maximum-entropy approximants, optimal-transportation mesh-free schemes, immersed-boundary methods, immersed finite elements, and cut finite elements all decouple, to different degrees, the description of geometry from the approximation space \cite{BelytschkoLuGu1994, LiuJunZhang1995, DuarteOden1996, BabuskaMelenk1997, BelytschkoParimiMoesSukumarUsui2003, Peskin2002, MittalIaccarino2005, Buhmann2003, Sukumar2004, Wendland2005, ArroyoOrtiz2006, SukumarWright2007, CyronArroyoOrtiz2009, BompadrePerottiCyronOrtiz2012, LiHabbalOrtiz2010, LiStalzerOrtiz2014, WeissenfelsWriggers2018, WangLiaoFanFanStainierLiLi2020, BurmanClausHansboLarsonMassing2015}. Their advantages are especially apparent in fracture, contact, evolving topology, immersed interfaces, image- or CAD-defined domains, and other geometry-rich settings in which repeated meshing is costly or unreliable. In their conventional use, however, these methods remain \emph{solvers}: they construct trial, test, quadrature, stabilization, or coupling mechanisms, and the governing variational or balance equations are then solved on the resulting discrete spaces.

The aim here is different. We use mesh-free and immersed ideas as a way to construct features and manufactured data for \emph{learning} a map associated with a boundary-value problem. The object of approximation is not a displacement field at one load, but a scalar value functional: the negative of the minimum potential energy as a function of the applied generalized load and the prescribed Dirichlet displacement. This places the work close in spirit to operator learning, where neural architectures approximate maps between function spaces \cite{LuJinPangZhangKarniadakis2021,KovachkiLiLiuAzizzadenesheliBhattacharyaStuartAnandkumar2023}, but the target here is a variational value map whose supporting slopes recover the mechanical fields. The scalar value is therefore not a loss of information: its load subgradients contain equilibrium displacements and its Dirichlet supergradients contain boundary reactions.

The geometry of the data space dictates the appropriate architecture. If the Dirichlet datum is fixed, the remaining datum enters the potential energy linearly through the load. The value functional is then a supremum of affine functions of the load datum and is convex. A single minmax envelope, equivalently a maximum of affine neurons, is structurally natural \cite{ContiOrtiz2025}. Minmax units were introduced as trainable piecewise-affine activations \cite{GoodfellowWardeFarleyMirzaEtAl2013}; a maximum of affine functions represents a convex piecewise-affine function. Generic nonconvex piecewise-affine approximation can be achieved by differences of convex minmax envelopes \cite{Hartman1959, WangSun2005}, and related minmax and polyaffine minmax constructions have been used in variational material identification \cite{ContiOrtiz2025}.

When the prescribed displacement is itself an input, the admissible class moves with the datum. Joint convexity in the full datum is then lost. However, the value functional is not structurally arbitrary: it is convex in the load and concave in the Dirichlet datum. The corresponding full-data architecture is therefore a saddle minmax, a finite zero-sum game in which displacement atoms generate the convex load slopes and reaction atoms generate the concave Dirichlet slopes. Its elementary payoff has a variational Lagrangian form in terms of Dirichlet reactions, and the coupling coefficients are determined by mechanics, rather than by arbitrary trainable intercepts. Removing the reaction side recovers the single minmax envelope of the fixed-Dirichlet or balanced full-Neumann load-only problem.

The construction also borrows from the method of manufactured solutions (MMS). In traditional MMS, an exact field is chosen and source terms and boundary data are computed that make it an exact solution. The resulting data is then used for code verification, order-of-accuracy studies, validation protocols, or uncertainty-quantification workflows \cite{SteinbergRoache1985, OberkampfTrucano2002, Roy2005, RoyOberkampf2011}. In the present paper, manufactured solutions are used differently. They provide labeled samples and exact supporting slopes for training a structure-preserving surrogate of the value functional. Thus, instead of using a manufactured solution to test a solver, we use manufactured displacement--reaction pairs to populate the atoms of a saddle minmax learner.

The paper is organized around three main steps. Section~2 formulates the mixed-boundary problem and collects the value-map theorem: well-posedness, local Lipschitz continuity, convexity in the load, concavity in the Dirichlet datum, and mechanical sensitivity. Section~3 develops the saddle minmax architecture: the exact saddle representation, manufactured support pairs, finite games, mechanical readouts, and uniform convergence. Section~4 gives one concrete implementation by Fourier atoms and cell-center quadrature. Section~5 presents an illustrative dam example, and the final section summarizes the framework.

\section{Variational value map for mixed data}

We begin by fixing the mechanical and functional setting and notation used throughout the analysis.

Let $\Omega\subset\mathbb R^d$ be the reference configuration of an elastic solid. The solid is acted upon by an externally applied load and by a prescribed displacement on a Dirichlet part of the boundary. Neumann tractions, when represented classically, are included among the external loads on $\Gamma_N$; tractions on $\Gamma_D$ are not prescribed independently, but appear as reactions. The larger dual space $X^*$ will nevertheless be used as an ambient generalized-load space.

We use the following notation consistently. The stored energy is denoted by $\mathcal E$, the trace lifting by $\Lift$, the potential energy by $\mathcal F$, and the value map by $\mathcal G$. The trace on $\Gamma_D$ is $T_D$, its adjoint is $T_D^*$, and duality pairings are indicated by their underlying spaces when this prevents ambiguity. The external-load space is denoted by ${L}_{\rm ext}$, while $X^*$ remains the ambient dual space used for analysis and for augmented manufactured data.

\subsection{Assumptions}

The following assumptions underpin subsequent analysis. 

\subsubsection{Data regularity assumptions}

\begin{itemize}[leftmargin=2.2em]
\item[\textnormal{(A1)}] \textbf{Domain regularity.}
$\Omega$ is bounded, connected, open, and Lipschitz, with outer unit normal $\nu:\partial\Omega\to S^{d-1}$.

\item[\textnormal{(A2)}] \textbf{Boundary decomposition.}
The sets $\Gamma_D,\Gamma_N\subseteq\partial\Omega$ are disjoint relatively open subsets of the boundary, their closures cover the boundary, $\overline\Gamma_D\cup\overline\Gamma_N=\partial\Omega$, the interface has zero surface measure, $\mathcal H^{d-1}(\overline\Gamma_D\cap\overline\Gamma_N)=0$, and $\mathcal H^{d-1}(\Gamma_D)>0$. Traces are understood in the Sobolev trace sense.

\item[\textnormal{(A3)}] \textbf{Data spaces and load convention.}
Let $p,q\in(1,\infty)$ satisfy $1/p+1/q=1$, and let
\begin{equation}
    X:=W^{1,p}(\Omega;\mathbb R^d),
    \quad
    G_D:=W^{1-1/p,p}(\Gamma_D;\mathbb R^d)
    =W^{1/q,p}(\Gamma_D;\mathbb R^d).
\end{equation}
The ambient generalized-load space is $X^*$. An ambient \emph{datum} is $z:=(f,g)\in Z$, where $Z:=X^*\times G_D$ is the datum space. Similarly, a displacement--reaction pair is denoted $w:=(u,R)\in{Y}$, where ${Y}:=X\times G_D^*$. We denote the Dirichlet trace map by
\begin{equation}\label{eq:dirichlet-trace-map}
    T_D:X\to G_D,
    \quad
    T_Du:=\operatorname{Tr}u|_{\Gamma_D},
\end{equation}
and denote its adjoint by $T_D^*:G_D^*\to X^*$,
\begin{equation}\label{eq:dirichlet-trace-adjoint}
    \langle T_D^*R,u\rangle_{X^*,X}:=R[T_D u],
    \quad 
    R\in G_D^*,\ u\in X ,
\end{equation}
where $R \in T_D^*G_D^*$ is to be interpreted as Dirichlet reactions or augmented forces, not as applied loads on $\Gamma_D$.

The physically prescribed external-load space is a closed linear subspace ${L}_{\rm ext} \subset X^*$, not containing independently prescribed Dirichlet-boundary tractions. The corresponding physical data space is ${L}_{\rm ext} \times G_D\subset Z$. The precise choice of ${L}_{\rm ext}$ is part of the load model. For instance, we may consider generalized loads of the form 
\begin{equation}\label{eq:classical-load-representation}
    f(\xi)=\int_\Omega b\cdot \xi\,dx+
    \langle h,\operatorname{Tr}\xi\rangle_{\Gamma_N},
    \quad \xi\in X,
\end{equation}
for suitable choices of $b$ and $h$ spanning a closed linear subspace of $X^*$. The estimates and saddle formulae below are stated on the larger space $Z$; restricting $f$ to ${L}_{\rm ext}$ gives the physical mixed-boundary problem. 

\item[\textnormal{(A4)}] \textbf{Trace lifting.}
There is a bounded linear right inverse $\Lift:G_D\to X$ of the Dirichlet trace, i.e. $T_D(\Lift g)=g$ on $\Gamma_D$.
\end{itemize}

We note that assumption (A4) is a standard consequence of the Sobolev trace theorem (e.g., \cite{AdamsFournier2003,McLean2000}) provided that $\Gamma_D$ is an admissible Lipschitz portion of the boundary. However, for an arbitrary relatively open subset $\Gamma_D\subset\partial\Omega$ this lifting property need not follow from the basic Lipschitz regularity of $\Omega$ alone, and it should therefore be retained as an explicit assumption. Related mixed-boundary regularity and trace issues are treated, for instance, in \cite{Groeger1989,AuscherBadrHallerDintelmannRehberg2015,HallerDintelmannJonssonKneesRehberg2016}.

\subsubsection{Energy-density assumptions}

Let $\mathbb{S}^d$ denote the space of symmetric $d\times d$ matrices, and let $W:\Omega\times\mathbb{S}^d\to[0,\infty)$ be the stored-energy density. We assume:

\begin{itemize}[leftmargin=2.2em]
\item[\textnormal{(A5)}] \textbf{Carath\'eodory regularity.}
The map $x\mapsto W(x,A)$ is measurable for every $A\in\mathbb{S}^d$, and $A\mapsto W(x,A)$ is continuous for a.e. $x\in\Omega$.

\item[\textnormal{(A6)}] \textbf{Strict convexity.}
For a.e. $x\in\Omega$, the map $A\mapsto W(x,A)$ is strictly convex on $\mathbb{S}^d$.

\item[\textnormal{(A7)}] \textbf{$p$-coercivity and $p$-growth.}
There are constants $\alpha,\beta>0$ and a function $a\in L^1(\Omega)$ such that, for a.e. $x\in\Omega$ and every $A\in\mathbb{S}^d$,
\begin{equation}\label{eq:p-growth}
    \alpha |A|^p-a(x)\leq W(x,A)\leq \beta(1+|A|^p)+a(x).
\end{equation}
The exponent $p\in(1,\infty)$ is arbitrary; the Hookean examples correspond to $p=2$. Continuity properties of the associated nonlinear superposition operators may be viewed in the standard Nemytskij-operator framework \cite{KrasnoselskiiRutickii1961,AppellZabrejko1990,Goebel1992}.

\item[\textnormal{(A8)}] \textbf{Local Lipschitz growth in the strain.}
There exists $C_W>0$ such that, for a.e. $x\in\Omega$ and every $A_1,A_2\in\mathbb{S}^d$,
\begin{equation}\label{eq:lip-growth}
    |W(x,A_1)-W(x,A_2)|
    \leq C_W\bigl(1+|A_1|^{p-1}+|A_2|^{p-1}\bigr)|A_1-A_2|.
\end{equation}

\item[\textnormal{(A9)}] \textbf{Uniform monotonicity of the stored energy.} The map $A \mapsto W(x,A)$ is continuously differentiable for a.e. $x \in \Omega$. With $S(x,A) := D_A W(x,A) \in \mathbb{S}^d$, there is a constant $C>0$ such that, for a.e. $x\in\Omega$ and all $A,B\in \mathbb{S}^d$,
\begin{equation}
    \big(S(x,A)-S(x,B)\big):(A-B)
    \ge 
    C \, |A-B|^p .
\end{equation}
\end{itemize}

\subsection{Problem statement}

For fixed $g\in G_D$ the admissible displacement class is
\begin{equation}
    X_g:=\{u\in X:\operatorname{Tr}u=g\text{ on }\Gamma_D\}.
\end{equation}
The lifting $u=\Lift g+v$ reduces the moving affine space to a fixed linear space,
\begin{equation}\label{eq:lifting}
    V:=\{v\in X:\operatorname{Tr}v=0\text{ on }\Gamma_D\}.
\end{equation}
Since $\mathcal H^{d-1}(\Gamma_D)>0$, Korn's inequality on $V$ gives \cite{Ciarlet1988}
\begin{equation}\label{eq:korn-mixed}
    \|v\|_{W^{1,p}(\Omega)}\leq C_K\|e(v)\|_{L^p(\Omega)}.
\end{equation}
For $z=(f,g)\in Z$ and $u\in X_g$, define the potential and stored energies as
\begin{equation}\label{eq:potential}
    \mathcal F(u;z):=\mathcal E(u)-\langle f,u\rangle_{X^*,X},
    \quad
    \mathcal E(u):=\int_\Omega W(x,e(u))\,dx.
\end{equation}
The corresponding lifted functional on the fixed space $V$ is
\begin{equation}\label{eq:lifted-G}
    \mathcal F_0(v;z)
    :=\mathcal F(\Lift g+v;z)
    =\mathcal E(\Lift g+v)-\langle f,\Lift g+v\rangle_{X^*,X}.
\end{equation}
The term $\langle f,\Lift g\rangle_{X^*,X}$ is retained. This is the lift-load term that makes the reduced problem have the same value as the unreduced potential; for $f\in{L}_{\rm ext}$ this is the physical potential energy.

\begin{remark}[Covariance under superposed infinitesimal rigid motions] \label{ZrRyFo}
{\rm The space of infinitesimal rigid displacements is
\begin{equation}\label{eq:rigid-space}
    \Rrig:=\{r(x)=a+Sx:
    a\in\mathbb R^d,
    S\in\mathbb R^{d\times d},\ S^T=-S\}.
\end{equation}
Since $e(r)=0$ for every $r\in\Rrig$, the stored energy is invariant under superposed infinitesimal rigid-body motions,
\begin{equation}\label{eq:rigid-invariance}
    \mathcal E(u+r)=\mathcal E(u),\quad r\in\Rrig.
\end{equation}
Let $r\in\Rrig$ and suppose the Dirichlet datum is changed from $g$ to $g+\operatorname{Tr}r$. Since $e(u+r)=e(u)$,
\begin{equation}
    \mathcal F(u+r;(f,g+\operatorname{Tr}r))
    =
    \mathcal F(u;z)-\langle f,r\rangle_{X^*,X} .
\end{equation}
If the load annihilates all rigid motions, then $\mathcal F(u;z)$ is invariant under the simultaneous superposition of $r$ on the displacement field and the Dirichlet datum.}
\hfill$\square$
\end{remark}

We assume the governing physical principle to be the \emph{principle of minimum potential energy}. Thus, for given data $z=(f,g)\in Z$, the problem consists equivalently of finding $u\in X_g$ that minimizes $\mathcal F(u;z)$; or finding $v\in V$ that minimizes $\mathcal F_0(v;z)$.

\subsection{Existence of solutions}

We record the direct-method existence statement on which the later value analysis rests.

\begin{proposition}[Existence of minimizers]\label{prop:existence}
Assume \textup{(A1)--(A7)}. Then, for every $z=(f,g)\in Z$, the functional $\mathcal F_0(\cdot;z)$ attains its minimum on $V$.
\end{proposition}

The proof is the standard coercivity--compactness argument, included to fix the role of the assumptions.

\begin{proof}
Fix $z=(f,g)\in Z$. Since $\Lift g\in X$, the value $\mathcal F_0(0;z)=\mathcal F(\Lift g;z)$ is finite. Let $(v_j)\subset V$ be a minimizing sequence and set $u_j:=\Lift g+v_j$. The load term satisfies
\begin{equation}
    |\langle f,u_j\rangle_{X^*,X}|
    \leq \|f\|_{X^*}\|u_j\|_X
    \leq C\|f\|_{X^*}\bigl(\|\Lift g\|_X+\|e(v_j)\|_{L^p(\Omega)}\bigr),
\end{equation}
where the last step uses Korn's inequality on $V$. Using the coercivity of $W$, the triangle inequality, and Young's inequality, we obtain constants $c_1,c_2,c_3>0$, depending on $z$ but not on $j$, such that
\begin{equation}
    \mathcal F_0(v_j;z)
    \geq
    c_1\|e(v_j)\|_{L^p(\Omega)}^p
    -c_2\bigl(1+\|e(v_j)\|_{L^p(\Omega)}\bigr)-c_3.
\end{equation}
Since $(v_j)$ is minimizing and $\mathcal F_0(0;z)<\infty$, the sequence $(e(v_j))$ is bounded in $L^p(\Omega)$. Korn's inequality gives boundedness of $(v_j)$ in $X$.

Because $p>1$, the space $X$ is reflexive. Thus, after passing to a subsequence, $v_j\rightharpoonup v$ weakly in $X$ for some $v\in X$. The trace operator is continuous and linear, so $V$ is a closed linear subspace of $X$, hence weakly closed. Therefore, $v\in V$. Moreover, $e(\Lift g+v_j)\rightharpoonup e(\Lift g+v)$ weakly in $L^p(\Omega;\mathbb{S}^d)$. By the Carath\'eodory regularity, convexity, and $p$-growth assumptions on $W$, the integral functional $\mathcal E$ is weakly lower semicontinuous \cite{Morrey1952,AcerbiFusco1984,Dacorogna2008}. The load term is weakly continuous because $f\in X^*$. Consequently,
\begin{equation}
    \mathcal F_0(v;z)\leq\liminf_{j\to\infty}\mathcal F_0(v_j;z)
    =\inf_{\psi\in V}\mathcal F_0(\psi;z).
\end{equation}
Thus, $v$ attains the minimum of $\mathcal F_0(\cdot;z)$ on $V$.
\end{proof}

The preceding existence statement becomes unique once strict convexity is invoked.

\begin{corollary}[Uniqueness under strict convexity]\label{cor:uniqueness}
Assume \textnormal{(A1)--(A7)}. Then, for every datum $z=(f,g)\in Z$, the minimizer of $\mathcal F_0(\cdot;z)$ in $V$ is unique. Equivalently, the equilibrium displacement $u_z=\Lift g+v_z\in X_g$ is unique.
\end{corollary}

This uniqueness gives a single equilibrium displacement, while the associated reaction may still be set-valued in nonsmooth settings.

\begin{proof}
By Proposition~\ref{prop:existence}, minimizers exist. Let $v_0,v_1\in V$ be two minimizers and set $u_i:=\Lift g+v_i$, $i=0,1$. Since the load term in $\mathcal F_0$ is affine in $v$, convexity of $W(x,\cdot)$ gives
\begin{equation}
    \mathcal F_0\left(\frac{v_0+v_1}{2};z\right)
    \leq \frac12\mathcal F_0(v_0;z)+\frac12\mathcal F_0(v_1;z).
\end{equation}
The right-hand side is the minimum value, so equality must hold. Strict convexity in the strain then implies $e(u_0)=e(u_1)$ a.e. Since $u_0-u_1=v_0-v_1\in V$, Korn's inequality gives $v_0=v_1$.
\end{proof}

The preceding corollary gives uniqueness of the equilibrium displacement for each datum. In later approximation arguments it will also be useful to know when equilibrium displacements associated with compact data classes form compact sets in the displacement space. This does not follow from uniqueness alone. The following strengthening records a standard sufficient condition: uniform monotonicity of the stored energy upgrades weak stability of minimizers to strong stability in $X$.

\begin{lemma}[Strong continuity and compactness of equilibrium displacements]
\label{lem:strong-continuity-compactness}
Assume {\rm (A1)--(A9)}. Then, the equilibrium map $Z \ni z=(f,g) \longmapsto u_z \in X$ is strongly continuous. Therefore, if $K\subset Z$ is compact, then $U_K:=\{u_z:z\in K\}$ is compact in $X$.
\end{lemma}

\begin{proof}
Let $z_n=(f_n,g_n)\to z=(f,g)$ in $Z=X^*\times G_D$, and write $u_n:=u_{z_n}=\Lift g_n+v_n$, $u:=u_z=\Lift g+v$, with $v_n,v\in V$. Since $(z_n)$ is bounded in $Z$, the coercivity estimate used in the proof of Proposition~\ref{prop:existence} gives a uniform bound for $(v_n)$ in $X$, and hence a uniform bound for $(u_n)$ in $X$. By reflexivity, after passing to a subsequence, $u_n \rightharpoonup \bar u$ weakly in $X$. The trace operator is continuous and $g_n\to g$ in $G_D$, so $T_D\bar u=g$. 

We now identify $\bar u$. For every $\xi\in X_g$, the comparison field $\xi_n:=\xi+\Lift(g_n-g)$ belongs to $X_{g_n}$ and satisfies $\xi_n\to\xi$ strongly in $X$. Since $u_n$ minimizes $E(\cdot)-\langle f_n,\cdot\rangle$ over $X_{g_n}$,
\begin{equation}
    E(u_n)-\langle f_n,u_n\rangle
    \le
    E(\xi_n)-\langle f_n,\xi_n\rangle .
\end{equation}
From the weak lower semicontinuity of $E$, the strong convergence $f_n\to f$ in $X^*$, the boundedness of $(u_n)$ in $X$, and $\xi_n\to\xi$ in $X$, we obtain
\begin{equation}
    E(\bar u)-\langle f,\bar u\rangle
    \le
    \liminf_{n\to\infty}
    \big(E(u_n)-\langle f_n,u_n\rangle\big)
    \le
    E(\xi)-\langle f,\xi\rangle .
\end{equation}
Since $\xi\in X_g$ is arbitrary, $\bar u$ minimizes the limiting problem. By Corollary~\ref{cor:uniqueness}, the minimizer is unique; hence $\bar u=u$. Therefore, every subsequence of $(u_n)$ has a further subsequence converging weakly to $u$, and consequently $u_n\rightharpoonup u$ weakly in $X$.

It remains to upgrade weak convergence to strong convergence. Under (A8)--(A9), the integral functional $\mathcal E$ is Fr\'echet differentiable on $X$, with
\begin{equation}
    \langle D\mathcal E(w),\eta\rangle_{X^*,X}
    =
    \int_\Omega S(x,e(w)):e(\eta)\,dx,
    \quad \eta\in X,
\end{equation}
and $D\mathcal E$ maps bounded subsets of $X$ into bounded subsets of $X^*$. The Euler equations are
\begin{equation}
    \langle DE(u_n)-f_n,\eta\rangle_{X^*,X}=0,
    \quad
    \langle DE(u)-f,\eta\rangle_{X^*,X}=0,
    \quad \eta\in V .
\end{equation}
Set $d_n:=v_n-v\in V$, $r_n:=\Lift(g_n-g)$. Then, $u_n-u=d_n+r_n$, $r_n\to0$ strongly in $X$. Subtracting the two Euler equations and testing with $d_n$ gives
\begin{equation}
    \langle DE(u_n)-DE(u),d_n\rangle_{X^*,X}
    =
    \langle f_n-f,d_n\rangle_{X^*,X}.
\end{equation}
Hence,
\begin{equation}
\begin{split}
    &
    \langle DE(u_n)-DE(u),u_n-u\rangle_{X^*,X}
    = \\ &
    \langle f_n-f,d_n\rangle_{X^*,X}
    +
    \langle DE(u_n)-DE(u),r_n\rangle_{X^*,X}.
\end{split}
\end{equation}
The first term tends to zero because $f_n\to f$ in $X^*$ and $(d_n)$ is bounded in $X$. The second term tends to zero because $r_n\to0$ strongly in $X$, while $(DE(u_n)-DE(u))$ is bounded in $X^*$ on bounded subsets of $X$, by the growth assumption {\rm (A8)}. Therefore,
\begin{equation}
    \langle DE(u_n)-DE(u),u_n-u\rangle_{X^*,X}\to0 .
\end{equation}
By {\rm (A9)},
\begin{equation}
\begin{split}
    &
    \langle DE(u_n)-DE(u),u_n-u\rangle_{X^*,X}
    = \\ &
    \int_\Omega
    \big(S(x,e(u_n))-S(x,e(u))\big):(e(u_n)-e(u))\,dx
    \ge \\ &
    C \, \|e(u_n-u)\|_{L^p(\Omega)}^p .
\end{split}
\end{equation}
Hence, $e(u_n-u)\to0$ strongly in $L^p(\Omega;\mathbb{S}^d)$. Since $d_n=u_n-u-r_n\in V$ and $r_n\to0$ in $X$, Korn's inequality on $V$ gives
\begin{equation}
\begin{split}
    \|d_n\|_X &\le C_K\|e(d_n)\|_{L^p(\Omega)}
    \\ & \le C_K\big(\|e(u_n-u)\|_{L^p(\Omega)}+\|e(r_n)\|_{L^p(\Omega)}\big)
    \to0 .
\end{split}
\end{equation}
Consequently,
\begin{equation}
    \|u_n-u\|_X\le \|d_n\|_X+\|r_n\|_X\to0 ,
\end{equation}
which proves strong continuity of $z\mapsto u_z$. 

If $K\subset Z$ is compact, then $U_K$ is the continuous image of the compact set $K$ under the map $z\mapsto u_z$. Hence, $U_K$ is compact in $X$.
\end{proof}

We note that the general value-functional properties below, including finiteness, lower semicontinuity, local Lipschitz continuity, and the convex--concave structure in the data, are based on hypotheses {\rm (A1)--(A8)}, whereas the compactness conclusion requires the additional hypothesis {\rm (A9)}. 

\subsection{The full-data value functional}

With the boundary-value problem in place, we now pass from equilibrium fields to the scalar value functional that will be learned.

Define the \emph{value functional} $\mathcal G:Z\to\mathbb R\cup\{+\infty\}$ as
\begin{equation}\label{eq:value}
    \mathcal G(z)
    :=
    -
    \inf_{v\in V}\mathcal F_0(v;z)
    =
    \sup_{v\in V}\{-\mathcal E(\Lift g+v)+\langle f,\Lift g+v\rangle_{X^*,X}\} .
\end{equation}
Thus, $\mathcal G(z)$ is the negative minimum potential energy attained at datum $z$. By Corollary~\ref{cor:uniqueness}, this minimum is attained by a unique $v_z\in V$. We write
\begin{equation}\label{eq:unique-equilibrium}
    u_z:=\Lift g+v_z\in X_g
\end{equation}
for the corresponding equilibrium displacement. By Remark~\ref{ZrRyFo}, it follows that
\begin{equation}
    \mathcal G(f,g+\operatorname{Tr}r)
    =
    \mathcal G(z)+\langle f,r\rangle_{X^*,X}.
\end{equation} 
Thus, if the load $f$ annihilates all rigid motions, $\mathcal G(z)$ is invariant under superposed rigid-body motions on the Dirichlet datum.

The next result collects the basic analytic properties needed for approximation.

\begin{proposition}[Properties of the value functional]\label{prop:value-properties}
Assume \textnormal{(A1)--(A8)}. Then, $\mathcal G:Z\to\mathbb R$ is finite-valued, proper, strongly lower semicontinuous, and locally Lipschitz continuous.
\end{proposition}

We give the details, since the estimates also guide the later stability arguments.

\begin{proof}
Testing \eqref{eq:value} with $v=0$ gives $\mathcal G(z)>-\infty$. Conversely, for $u=\Lift g+v$,
\begin{equation}
    |\langle f,u\rangle_{X^*,X}|
    \leq |f|_{X^*}|u|_X,
    \quad
    |u|_X\leq C\bigl(|g|_{G_D}+|e(v)|_{L^p}\bigr).
\end{equation}
The coercivity of $W$ and Young's inequality give a uniform upper bound for $-\mathcal F_0(v;z)$; hence $\mathcal G(z)<+\infty$. Thus $\mathcal G$ is finite-valued and proper.

For fixed $v\in V$, the map
\begin{equation}
    z=(f,g)\longmapsto
    -\mathcal E(\Lift g+v)+\langle f,\Lift g+v\rangle_{X^*,X}
\end{equation}
is continuous on $Z$. Indeed, if $z_j=(f_j,g_j)\to z=(f,g)$ in $Z$, then $\Lift g_j\to \Lift g$ in $X$, the load terms converge by the duality pairing, and the energy term converges by \eqref{eq:lip-growth}. Since $\mathcal G$ is the pointwise supremum of these continuous functions over $v\in V$, it is strongly lower semicontinuous.

It remains to prove local Lipschitz continuity. Fix $R>0$ and suppose
\begin{equation}
    |z_i|_Z\leq R,\qquad z_i=(f_i,g_i),\qquad i=1,2.
\end{equation}
Let $u_i=\Lift g_i+v_i\in X_{g_i}$ be the corresponding equilibrium displacement. Comparing the minimum at $u_i$ with the admissible field $\Lift g_i$ and using coercivity, Korn's inequality on $V$, the boundedness of $\Lift$, and Young's inequality gives
\begin{equation}
    |u_i|_X\leq M_R,\qquad i=1,2,
\end{equation}
for a constant $M_R$ depending only on $R$ and on the constants in the assumptions. Set
\begin{equation}
    \delta f:=f_2-f_1,\qquad \delta g:=g_2-g_1.
\end{equation}
Since the admissible spaces move with the Dirichlet datum, we compare across the two spaces by lifting the boundary mismatch. Define
\begin{equation}
    \widehat u_1:=u_1+\Lift \delta g .
\end{equation}
Then $T_D\widehat u_1=g_2$, hence $\widehat u_1\in X_{g_2}$. Using $\widehat u_1$ as a competitor for the datum $z_2$ gives
\begin{equation}
\begin{split}
    \mathcal G(z_1)-\mathcal G(z_2)
    & \leq
    \bigl(\langle f_1,u_1\rangle_{X^*,X}-\mathcal E(u_1)\bigr)
    -
    \bigl(\langle f_2,\widehat u_1\rangle_{X^*,X}-\mathcal E(\widehat u_1)\bigr)  
    \\ & =
    \langle f_1-f_2,u_1\rangle_{X^*,X}
    -
    \langle f_2,\Lift\delta g\rangle_{X^*,X}
    +
    \mathcal E(\widehat u_1)-\mathcal E(u_1).
\end{split}
\end{equation}
The first two terms are bounded by
\begin{equation}
    |\langle f_1-f_2,u_1\rangle_{X^*,X}|
    +
    |\langle f_2,\Lift\delta g\rangle_{X^*,X}|
    \leq
    C_R\bigl(|f_1-f_2|_{X^*}+|g_1-g_2|_{G_D}\bigr).
\end{equation}
Moreover,
\begin{equation}
    |\widehat u_1|_X
    \leq
    M_R+|\Lift|_{\mathcal L(G_D,X)}|g_2-g_1|_{G_D},
\end{equation}
and, after increasing the constant on the bounded data ball if necessary, the local Lipschitz growth assumption \eqref{eq:lip-growth} gives
\begin{equation}
    |\mathcal E(\widehat u_1)-\mathcal E(u_1)|
    \leq
    C_R|\Lift\delta g|_X
    \leq
    C_R|g_1-g_2|_{G_D}.
\end{equation}
Therefore,
\begin{equation}
    \mathcal G(z_1)-\mathcal G(z_2)
    \leq
    C_R|z_1-z_2|_Z .
\end{equation}

Interchanging the roles of $z_1$ and $z_2$, and using the competitor
\begin{equation}
    \widehat u_2:=u_2+\Lift(g_1-g_2)\in X_{g_1},
\end{equation}
gives the reverse estimate
\begin{equation}
    \mathcal G(z_2)-\mathcal G(z_1)
    \leq
    C_R|z_1-z_2|_Z .
\end{equation}
Consequently,
\begin{equation}
    |\mathcal G(z_1)-\mathcal G(z_2)|
    \leq
    C_R|z_1-z_2|_Z ,
\end{equation}
which proves local Lipschitz continuity.
\end{proof}

The saddle geometry of the value map is made precise in the following proposition.

\begin{proposition}[Separate convexity--concavity of the value functional]\label{prop:value-convex-concave}
Assume \textnormal{(A1)--(A7)}. Let $z = (f,g) \in Z$. Then, the mixed-boundary value functional $\mathcal G(z)$ is convex in the load variable $f$ and concave in the Dirichlet variable $g$, i.e., for every fixed $g\in G_D$, the map $X^*\ni f\mapsto \mathcal G(z)$ is convex, and, for every fixed $f\in X^*$, the map $G_D\ni g\mapsto \mathcal G(z)$ is concave. 
\end{proposition}

This structural observation is the point of departure for the minmax construction below.

\begin{proof}
Fix $g\in G_D$. For each fixed $v\in V$ the right-hand side of \eqref{eq:value} is an affine function of $f$. Hence, $f\mapsto\mathcal G(z)$ is a supremum of affine functions and is therefore convex.

Next, fix $f\in X^*$, and write
\begin{equation}
    J_f(u):=\mathcal E(u)-\langle f,u\rangle_{X^*,X}.
\end{equation}
The functional $J_f$ is convex on $X$, because $\mathcal E$ is convex and the load term is linear. Let $g_0,g_1\in G_D$, $\theta\in[0,1]$, and set $g_\theta:=(1-\theta)g_0+\theta g_1$. By Proposition~\ref{prop:existence}, choose minimizers $u_i\in X_{g_i}$ of $J_f$ over $X_{g_i}$, $i=0,1$. Then,
\begin{equation}
    u_\theta:=(1-\theta)u_0+\theta u_1\in X_{g_\theta}.
\end{equation}
Consequently,
\begin{align}
    \inf_{u\in X_{g_\theta}}J_f(u)
    &\leq J_f(u_\theta) 
    \leq (1-\theta)J_f(u_0)+\theta J_f(u_1) \\
    &=(1-\theta)\inf_{u\in X_{g_0}}J_f(u)
      +\theta\inf_{u\in X_{g_1}}J_f(u).
\end{align}
Multiplication by $-1$ gives
\begin{equation}
    \mathcal G(f,g_\theta)
    \geq (1-\theta)\mathcal G(f,g_0)+\theta\mathcal G(f,g_1),
\end{equation}
which proves concavity in $g$.
\end{proof}

\begin{example}[An elementary truss example] \label{oE9UI9}
{\rm
\begin{figure}[h!]
\centering
\begin{subfigure}[b]{0.49\textwidth}
\centering
\includegraphics[width=\textwidth]{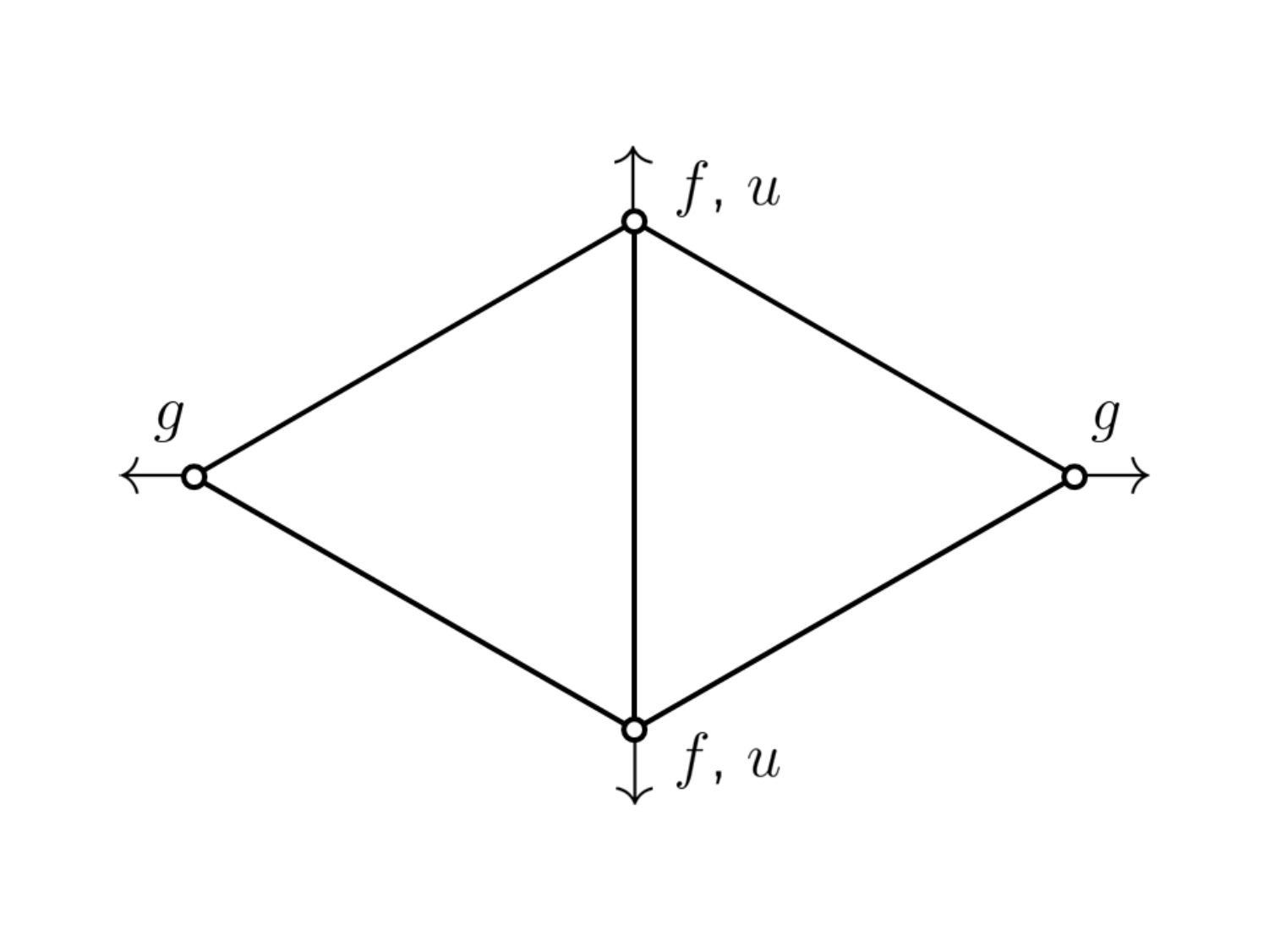}
\caption{Truss structure} \label{fig:Truss1}
\end{subfigure}
\hfill
\begin{subfigure}[b]{0.49\textwidth}
\centering
\includegraphics[width=\textwidth]{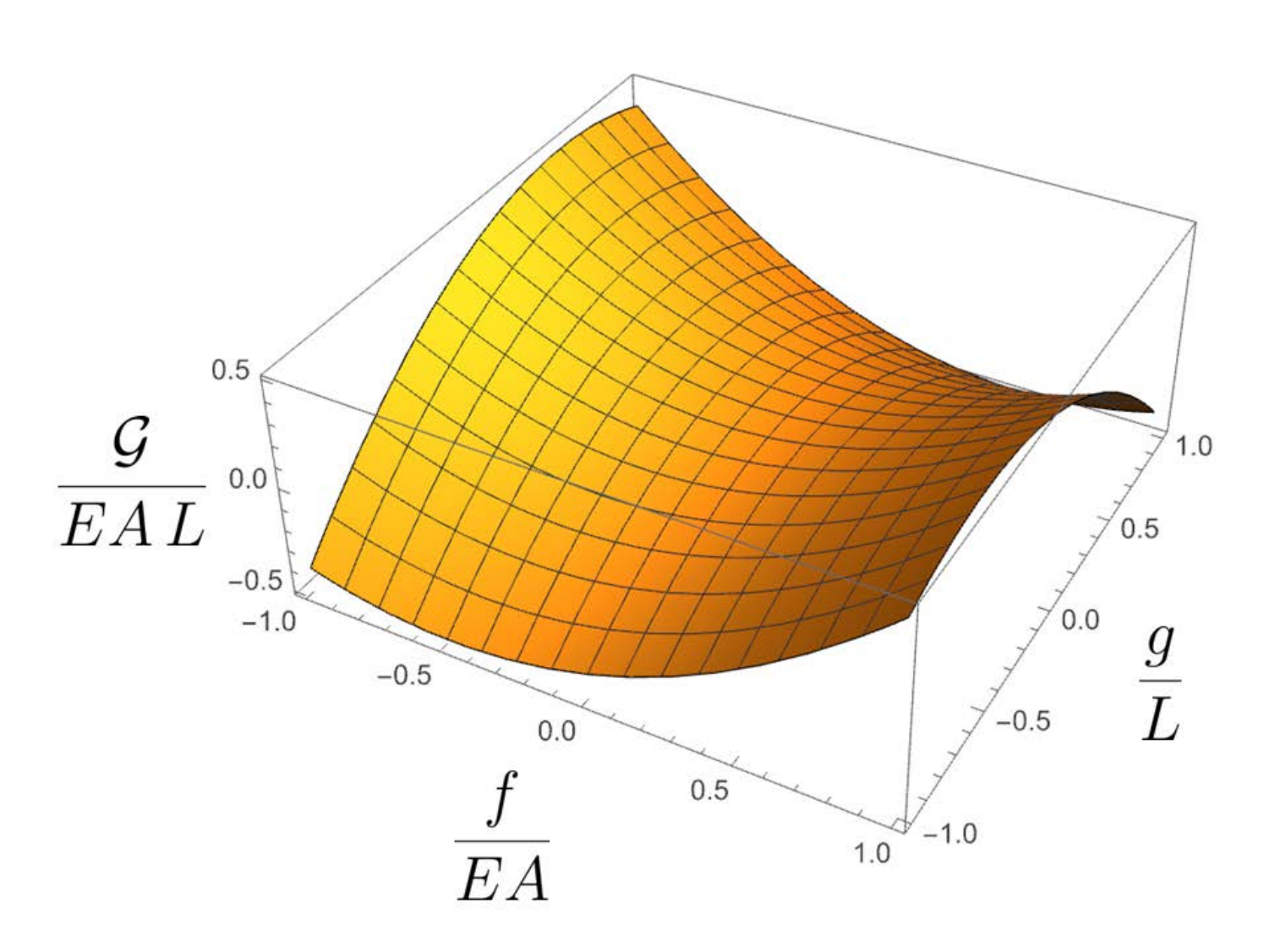}
\caption{Value function} \label{fig:Truss2}
\end{subfigure}
\caption{Simple symmetric truss example. (a) Lines represent bars of length $L$ and stiffness $EA/L$, open circles represent articulated joints, $f$ is an applied force, $g$ is a prescribed displacement, and $u$ is the displacement degree of freedom. (b) Value function $\mathcal G=-\min\mathcal F$, showing convexity in $f$ and concavity in $g$.} \label{fig:Truss}
\end{figure}

For the truss in Fig.~\ref{fig:Truss}(a), with $z=(f,g)$, direct minimization of
\begin{equation}
    \mathcal F(u;z)
    =
    \frac{{EA} L}{2} \Big(\frac{2 u}{L}\Big)^2 
    +
    4 \frac{{EA} L}{2} \Big(\frac{u+g}{2L}\Big)^2
    - 2 f u
\end{equation}
gives
\begin{equation}
    u_z=\frac{2 f L}{5 {EA}}-\frac{g}{5},
    \qquad
    \mathcal G(f,g)
    =
    \frac{2 f^2 L}{5 {EA}}
    -
    \frac{2 {EA} g^2}{5 L}
    -
    \frac{2 f g}{5} .
\end{equation}
Thus the value is convex in $f$ and concave in $g$, as predicted by Proposition~\ref{prop:value-convex-concave}.
} \hfill$\square$
\end{example}

\subsection{Sensitivity with respect to the data}\label{sec:sensitivity}

The value functional is locally Lipschitz and convex--concave in the separate data variables. Thus, its mechanical readout does not require classical differentiability. The appropriate objects are supporting slopes: load subgradients for the convex variable and Dirichlet supergradients for the concave variable. This places the construction within standard nonsmooth convex analysis, envelope sensitivity, and variational perturbation theory \cite{Danskin1967, Rockafellar1970, EkelandTemam1999, RockafellarWets1998, BonnansShapiro2000, HiriartUrrutyLemarechal1993}.

For $z=(f,g)\in Z$, define the load subdifferential and Dirichlet superdifferential as
\begin{subequations} \label{OtNsVO}
\begin{align}
    \partial_f\mathcal G(z)
    &:=\left\{u\in X \,:\,
    \mathcal G(\tilde f,g)\geq \mathcal G(f,g)
    +\langle \tilde f-f,u\rangle_{X^*,X} ,
    \;\forall\tilde f\in X^*\right\},
    \\
    \partial_g\mathcal G(z)
    &:=\left\{R\in G_D^* \,:\,
    \mathcal G(f,\tilde g)\leq \mathcal G(f,g)
    +R[\tilde g-g] ,
    \;\forall\tilde g\in G_D\right\}.
\end{align}
\end{subequations}
Let $u_z\in X_g$ be the unique equilibrium displacement. The convex subdifferential of the finite continuous convex functional $\mathcal E:X\to\mathbb R$ is
\begin{equation}
    \partial\mathcal E(u)
    :=\{P\in X^*:
    \mathcal E(\tilde u)\geq \mathcal E(u)+\langle P,\tilde u-u\rangle_{X^*,X}
    \;\text{for all }\tilde u\in X\}.
\end{equation}
At the constrained minimizer $u_z$, the Fermat rule and the standard convex subdifferential sum rule applied to
$\mathcal E(u)-\langle f,u\rangle_{X^*,X}+I_{X_g}(u)$ give
\begin{equation}
    f\in \partial\mathcal E(u_z)+N_{X_g}(u_z).
\end{equation}
Equivalently, there exists $P_z\in\partial\mathcal E(u_z)$ such that
\begin{equation}\label{eq:nonsmooth-equilibrium}
    \langle P_z-f,\eta\rangle_{X^*,X}=0,
    \quad \eta\in V .
\end{equation}
Such a $P_z$ is an admissible internal force at equilibrium. For any such $P_z$, define the Dirichlet reaction as
\begin{equation}\label{eq:Dirichlet-reaction}
    R_z[\eta]
    :=\langle f-P_z,\Lift\eta\rangle_{X^*,X},
    \quad \eta\in G_D .
\end{equation}
This definition is independent of the chosen lifting: if two liftings of $\eta$ differ, their difference lies in $V$, and \eqref{eq:nonsmooth-equilibrium} applies. Thus, $R_z\in G_D^*$ and $f=P_z+T_D^*R_z \in X^*$. In the smooth case $P_z=D\mathcal E(u_z)$, and \eqref{eq:Dirichlet-reaction} reduces to the usual weak reaction formula. If the load is classical and the fields are smooth, integration by parts gives
\begin{equation}
    R_z[\eta]=-
    \int_{\Gamma_D}(\sigma_z\nu)\cdot\eta\,ds,
\end{equation}
with the sign convention used throughout this paper.

The preceding construction identifies the reaction associated with any admissible internal force at equilibrium. We now record the corresponding sensitivity statement for the value functional. The result is nonsmooth in form: the equilibrium displacement is a supporting slope in the load variable, while the Dirichlet reaction is a supporting slope, with the opposite inequality, in the prescribed-displacement variable.

\begin{proposition}[Nonsmooth sensitivity and differentiable special case]\label{prop:direct-sensitivity}
Assume \textnormal{(A1)--(A7)}. Let $z=(f,g)\in Z$, let $u_z\in X_g$ be the unique equilibrium displacement, and let $P_z\in\partial\mathcal E(u_z)$ satisfy \eqref{eq:nonsmooth-equilibrium}. Define $R_z$ by \eqref{eq:Dirichlet-reaction}. Then,
\begin{equation}\label{eq:nonsmooth-sensitivity-inclusion}
    u_z\in \partial_f\mathcal G(z),
    \quad
    R_z\in \partial_g\mathcal G(z).
\end{equation}
Equivalently,
\begin{align}
    \mathcal G(\tilde f,g)
    &\geq \mathcal G(f,g)+\langle \tilde f-f,u_z\rangle_{X^*,X},\label{eq:load-support}\\
    \mathcal G(f,\tilde g)
    &\leq \mathcal G(f,g)+R_z[\tilde g-g],\label{eq:dirichlet-support}
\end{align}
for all $\tilde f\in X^*$ and all $\tilde g\in G_D$.

If, in addition, $\mathcal E$ is Fr\'echet differentiable at $u_z$ and $\mathcal G$ is Fr\'echet differentiable at $z$, then $P_z=D\mathcal E(u_z)$ and, for every $\delta z=(\delta f,\delta g)\in Z$,
\begin{equation}\label{eq:DG-formula}
    D\mathcal G(z)[\delta z]
    =\langle\delta f,u_z\rangle_{X^*,X}+R_z[\delta g].
\end{equation}
Thus, the classical sensitivity pair is $w_z=(u_z,R_z)\in Y$, with
\begin{equation}\label{eq:sensitivity-components}
    \frac{\delta\mathcal G}{\delta f}(z)=u_z\in X\simeq X^{**},
    \quad
    \frac{\delta\mathcal G}{\delta g}(z)=R_z\in G_D^* .
\end{equation}
\end{proposition}

\begin{proof}
For the load variable, using $u_z$ as an admissible comparison field for the datum $(\tilde f,g)$ gives
\begin{equation}
    \mathcal G(\tilde f,g)
    \geq
    \langle \tilde f,u_z\rangle_{X^*,X}-\mathcal E(u_z)
    =\mathcal G(f,g)+\langle \tilde f-f,u_z\rangle_{X^*,X},
\end{equation}
which proves \eqref{eq:load-support}.

For the Dirichlet variable, let $\tilde u\in X_{\tilde g}$. Convexity of $\mathcal E$ and $P_z\in\partial\mathcal E(u_z)$ give
\begin{equation}
    \langle f,\tilde u\rangle_{X^*,X}-\mathcal E(\tilde u)
    \leq
    \langle f,u_z\rangle_{X^*,X}-\mathcal E(u_z)
    +\langle f-P_z,\tilde u-u_z\rangle_{X^*,X}.
\end{equation}
Because $\tilde u-u_z$ has Dirichlet trace $\tilde g-g$ and $f-P_z$ annihilates $V$, the last term equals $R_z[\tilde g-g]$. Taking the supremum over $\tilde u\in X_{\tilde g}$ proves \eqref{eq:dirichlet-support}. 

If $\mathcal G$ is Fr\'echet differentiable at $z$, the supporting inequalities identify the derivative with the displayed linear functional, giving \eqref{eq:DG-formula}.
\end{proof}

Thus, the data derivative of the value functional, whenever it is single valued, has the expected mechanical interpretation: variations of the applied load are paired with the equilibrium displacement, and variations of the Dirichlet datum are paired with the boundary reaction. In the nonsmooth case the same statement persists at the level of supporting subgradients and supergradients.

\section{Saddle minmax architecture}\label{sec:saddle-minmax}

Proposition~\ref{prop:value-convex-concave} and Proposition~\ref{prop:direct-sensitivity} show that the full mixed-boundary value functional has a saddle geometry: it is convex with respect to the generalized load, concave with respect to the prescribed Dirichlet displacement, and its supporting slopes are mechanical displacement--reaction pairs. This section records the corresponding minmax architecture, with the Dirichlet constraint enforced by a dual reaction variable.

\subsection{Exact saddle representation}

Recall the Dirichlet trace map $T_D:X\to G_D$ and its adjoint $T_D^*:G_D^*\to X^*$ from \eqref{eq:dirichlet-trace-map}--\eqref{eq:dirichlet-trace-adjoint}. For $z=(f,g)\in Z$ and $w:=(u,R)\in Y$, define the mixed Lagrangian
\begin{equation}\label{eq:mixed-saddle-lagrangian}
    \mathcal L(w;z)
    :=
    \langle f,u\rangle_{X^*,X}+R[g-T_Du]-\mathcal E(u),
    \quad u\in X,
    \quad R\in G_D^* .
\end{equation}
The sign convention is chosen so that $R$ agrees with the reaction in \eqref{eq:Dirichlet-reaction}.

The mixed Lagrangian incorporates the Dirichlet constraint by duality rather than by restricting the displacement variable a priori. The following result shows that this unconstrained saddle formulation is exactly equivalent to the minimum-potential-energy value problem, and that the equilibrium displacement together with its reaction is a saddle point of the Lagrangian.

\begin{proposition}[Exact saddle representation]\label{prop:exact-saddle-representation}
Assume \textnormal{(A1)--(A7)}. Then, for every $z=(f,g)\in Z$,
\begin{equation}\label{eq:exact-saddle-representation}
    \mathcal G(z)
    =
    \sup_{u\in X}\inf_{R\in G_D^*}\mathcal L((u,R);z).
\end{equation}
Moreover, let $u_z$ be the equilibrium displacement and let $R_z$ be any reaction obtained from an admissible internal force $P_z\in\partial\mathcal E(u_z)$ through \eqref{eq:Dirichlet-reaction}. Then, $w_z=(u_z,R_z)$ is a saddle point:
\begin{equation}\label{eq:exact-saddle-inequality}
    \mathcal L(u,R_z;z)
    \leq
    \mathcal L(w_z;z)
    \leq
    \mathcal L(u_z,R;z)
\end{equation}
for every $u\in X$ and every $R\in G_D^*$, and
\begin{equation}
    \mathcal L(w_z;z)=\mathcal G(z).
\end{equation}
\end{proposition}

\begin{proof}
For fixed $u$, the map $R\mapsto R[g-T_Du]$ is linear. If $T_Du\ne g$, its infimum over $G_D^*$ is $-\infty$; if $T_Du=g$, the multiplier term vanishes. Hence, the extended-real inner infimum is finite exactly on $X_g$, and
\begin{equation}
    \sup_{u\in X}\inf_{R\in G_D^*}\mathcal L((u,R);z)
    =
    \sup_{u\in X_g}\{\langle f,u\rangle_{X^*,X}-\mathcal E(u)\}
    =\mathcal G(z).
\end{equation}
Since $T_Du_z=g$, the right inequality in \eqref{eq:exact-saddle-inequality} is an equality for every $R$. For the left inequality, convexity gives
\begin{equation}
    \mathcal E(u)\geq \mathcal E(u_z)+\langle P_z,u-u_z\rangle_{X^*,X}.
\end{equation}
Using $f=P_z+T_D^*R_z$ and $T_Du_z=g$ yields
\begin{equation}
    \mathcal L(u,R_z;z)
    \leq
    \langle f,u_z\rangle_{X^*,X}-\mathcal E(u_z)
    =\mathcal L(w_z;z),
\end{equation}
which proves the saddle inequality.
\end{proof}

Thus, the Dirichlet datum can be enforced entirely through the reaction variable. This exact saddle representation is the infinite-dimensional prototype for the finite minmax games introduced below: displacement atoms will approximate the maximizing side, while reaction atoms will approximate the minimizing side associated with the boundary constraint.

\begin{remark}[Full-Neumann and full-Dirichlet endpoint cases]
{\rm The mixed saddle architecture contains both endpoint boundary conditions as special cases. In the full-Neumann case, $\Gamma_N=\partial\Omega$ and $\Gamma_D=\emptyset$, there is no Dirichlet datum and no reaction multiplier. After passing to the quotient space $X/\Rrig$ by rigid motions, the exact representation reduces to the load-only convex envelope
\begin{equation}
    \mathcal G(f)
    =
    \sup_{u\in X/\Rrig}
    \{\langle f,u\rangle_{X^*,X}-\mathcal E(u)\},
\end{equation}
with $f\in (X/\Rrig)^*$ satisfying the usual balance condition. In the full-Dirichlet case, $\Gamma_D=\partial\Omega$ and $\Gamma_N=\emptyset$, the trace $T_D=T$ is the full boundary trace and the reaction multiplier acts on the entire boundary.} \hfill$\square$
\end{remark}

\subsection{Manufactured saddle data}\label{sec:manufactured-saddle-sensitivities}

The ambient generalized-load formulation makes it possible to manufacture displacement--reaction support pairs without solving a boundary-value problem for every training sample. The nonsmooth form of the construction uses a displacement atom $u$, an internal-force atom $P\in\partial\mathcal E(u)$, and a reaction atom $R\in G_D^*$.

\begin{proposition}[Manufactured data with prescribed reaction]\label{prop:manufactured-prescribed-reaction}
Assume \textnormal{(A1)--(A7)}. Let $u\in X$, $P\in\partial\mathcal E(u)$, and $R\in G_D^*$. Define
\begin{equation}\label{eq:manufactured-g-f-R}
    g_w:=T_Du,
    \quad
    f_w:=P+T_D^*R\in X^* .
\end{equation}
Then, $u$ is the unique equilibrium displacement for the ambient datum $z_w:=(f_w,g_w)\in Z$. The corresponding reaction contains $R$, and the exact value label is
\begin{equation}\label{eq:manufactured-label-with-reaction}
    \mathcal G(z_w)
    =\langle f_w,u\rangle_{X^*,X}-\mathcal E(u)
    =\langle P,u\rangle_{X^*,X}+R[T_Du]-\mathcal E(u).
\end{equation}
Moreover,
\begin{equation}\label{eq:manufactured-support-with-reaction}
    u\in\partial_f\mathcal G(z_w),
    \quad
    R\in\partial_g\mathcal G(z_w).
\end{equation}
At a datum where $\mathcal G$ is differentiable, this support pair gives the classical derivative formula \eqref{eq:DG-formula}.
\end{proposition}

\begin{proof}
Since $g_w=T_Du$, the field $u$ is admissible. For every $\varphi\in V$,
\begin{equation}
    \langle f_w,\varphi\rangle_{X^*,X}
    =\langle P,\varphi\rangle_{X^*,X}+R[T_D\varphi]
    =\langle P,\varphi\rangle_{X^*,X}.
\end{equation}
Thus, $P\in\partial\mathcal E(u)$ satisfies the affine-space optimality condition for the datum $z_w$, and $u$ is the unique minimizer. The reaction computed from \eqref{eq:Dirichlet-reaction} is, for every $\eta\in G_D$,
\begin{equation}
    \langle f_w-P,\Lift\eta\rangle_{X^*,X}=R[T_D\Lift\eta]=R[\eta].
\end{equation}
The value identity follows by evaluating the potential at the minimizer, and the supporting-slope statement follows from Proposition~\ref{prop:direct-sensitivity}.
\end{proof}

If the stored energy is differentiable at $u$, the canonical choice is $P=D\mathcal E(u)$, recovering the familiar manufactured load $f_w=D\mathcal E(u)+T_D^*R$. Choosing $R=0$ gives the usual manufactured-load sample. Nonzero reactions are physical external-load samples only when $f_w\in L_{\rm ext}$; otherwise they are augmented samples designed to expose the concave Dirichlet side of the value map.

\subsection{Finite saddle minmax games}\label{sec:finite-saddle-games}

Having obtained an exact saddle representation, we now restrict the two players to finite collections of atoms.

\subsubsection{From data-supported approximation to the finite saddle minmax game}

We begin by introducing the finite saddle architecture as a support-based approximation of the value functional on data space. Let
\begin{equation}
    z_\alpha=(f_\alpha,g_\alpha)\in Z,
    \quad 
    \alpha=1,\ldots,N,
\end{equation}
be sampled data points. By Proposition~\ref{prop:direct-sensitivity}, every datum $z_\alpha$
possesses mechanical supporting slopes
\begin{equation}
    u_\alpha\in \partial_f \mathcal{G}(z_\alpha),
    \quad
    R_\alpha\in \partial_g \mathcal{G}(z_\alpha),
\end{equation}
where $u_\alpha$ is the equilibrium displacement and $R_\alpha$ is the
corresponding Dirichlet reaction. Hence, the data-generated affine saddle function associated with $z_\alpha$ is
\begin{equation}
    \ell_\alpha(z)
    :=
    \mathcal{G}(z_\alpha)
    +
    \langle f-f_\alpha,u_\alpha\rangle_{X^*,X}
    +
    R_\alpha[g-g_\alpha],
    \quad z=(f,g)\in Z .
\end{equation}
This expression is affine in the load variable and affine in the Dirichlet datum. Its load slope and Dirichlet slope are precisely the mechanical pair $w_\alpha=(u_\alpha,R_\alpha)$. 

A finite architecture for $\mathcal{G}$ may therefore be built from such supporting data. Since $\mathcal{G}$ is convex in the load variable, the support functions are combined by a maximum on the load side. Since $\mathcal{G}$ is concave in the Dirichlet datum, they are combined by a minimum on the Dirichlet side. This leads naturally to a finite saddle envelope.

It remains to express the supports in a form that does not require the datum $z_\alpha$, the value $\mathcal{G}(z_\alpha)$, and the intercept to be stored independently. Proposition~\ref{prop:manufactured-prescribed-reaction} provides this reduction. Suppose that a displacement atom $u_i\in X$, an internal-force atom $P_i\in \partial\mathcal E(u_i)$, and a reaction atom $R_j\in G_D^*$ are prescribed. They induce the ambient datum
\begin{equation}
    g_i:=T_Du_i,
    \quad
    f_{ij}:=P_i+T_D^*R_j,
    \quad
    z_{ij}:=(f_{ij},g_i).
\end{equation}
For this datum, $u_i$ is the exact equilibrium displacement and $R_j$ is the corresponding Dirichlet reaction. Moreover,
\begin{equation}
    \mathcal{G}(z_{ij})
    =
    \langle f_{ij},u_i\rangle_{X^*,X}
    -
    \mathcal E(u_i).
\end{equation}
Substituting these identities into the supporting affine form gives
\begin{equation}
\begin{aligned}
    \ell_{ij}(z)
    & =
    \mathcal{G}(z_{ij})
    +
    \langle f-f_{ij},u_i\rangle_{X^*,X}
    +
    R_j[g-g_i] 
    \\ & =
    \langle f,u_i\rangle_{X^*,X}
    +
    R_j[g-T_Du_i]
    -
    \mathcal E(u_i).
\end{aligned}
\end{equation}
Thus, the elementary payoff associated with the pair of atoms $(u_i,R_j)$ is
\begin{equation}
    \Phi_{ij}(z)
    :=
    \langle f,u_i\rangle_{X^*,X}
    +
    R_j[g-T_Du_i]
    -
    \mathcal E(u_i) ,
\end{equation}
the coupling coefficient between the displacement atom $u_i$ and the reaction atom $R_j$ is $C_{ij}:=-R_j[T_Du_i]$, and the load-side intercept is $-\mathcal E(u_i)$. These quantities are determined by the mechanical energy and the Dirichlet trace pairing, rather than being independent trainable parameters.

Let $u_1,\ldots,u_I\in X$ be displacement atoms and let $R_1,\ldots,R_J\in G_D^*$ be reaction atoms. Denote by
\begin{subequations}
\begin{align}
    &
    \Delta_I
    :=
    \left\{
        p\in \mathbb R^I \,:\,
        p_i\ge 0, \;\, \sum_{i=1}^I p_i=1
    \right\},
    \\ & 
    \Delta_J
    :=
    \left\{
        q\in \mathbb R^J \,:\,
        q_j\ge 0, \;\, \sum_{j=1}^J q_j=1
    \right\},
\end{align}
\end{subequations}
the corresponding probability simplices. The finite saddle minmax approximation is then defined as the value of the finite zero-sum game
\begin{equation} \label{eq:saddle-minmax-finite}
    \mathcal{G}_{IJ}(z)
    :=
    \max_{p\in\Delta_I}
    \min_{q\in\Delta_J}
    \sum_{i=1}^I
    \sum_{j=1}^J
    p_iq_j\Phi_{ij}(z).
\end{equation}
The game is set forth by a finite family of data-generated affine saddle functions, with Propositions~\ref{prop:direct-sensitivity} and~\ref{prop:manufactured-prescribed-reaction} identifying those supports with mechanical displacement--reaction atoms. The architecture may be viewed either as a support-envelope approximation built from sampled data $z_{ij}$, or equivalently as a saddle minmax game built directly from the atom dictionaries $\{u_i\}_{i=1}^I$ and $\{R_j\}_{j=1}^J$.

\subsubsection{Properties and mechanical readout}

The finite game just defined inherits its structure from the exact saddle representation: the displacement atoms act on the convex load side, while the reaction atoms act on the concave Dirichlet side. We first record that this discretization preserves the fundamental convex--concave geometry of the full value functional.

\begin{proposition}[Shape preservation of the saddle minmax]\label{prop:saddle-minmax-shape}
For every finite collection $(u_i)_{i=1}^I\subset X$ and $(R_j)_{j=1}^J\subset G_D^*$, the function $\mathcal G_{IJ}:Z\to\mathbb R$ defined by \eqref{eq:saddle-minmax-finite} is finite, continuous, convex in $f$, and concave in $g$. The minimax identity
\begin{equation}\label{eq:finitely-dimensional-minimax}
    \max_{p\in\Delta_I}\min_{q\in\Delta_J}
    \sum_{i,j}p_iq_j\Phi_{ij}(z)
    =
    \min_{q\in\Delta_J}\max_{p\in\Delta_I}
    \sum_{i,j}p_iq_j\Phi_{ij}(z)
\end{equation}
holds for every $z\in Z$.
\end{proposition}

\begin{proof}
The payoff is continuous and affine in the datum for every fixed $(p,q)$, and the strategy sets are compact finite-dimensional simplices. The minimax identity is the finite-dimensional form of minimax duality \cite{vonNeumannMorgenstern1944,Sion1958}. For fixed $g$, the value is a maximum of affine functions of $f$; for fixed $f$, using the reversed minimax form, it is a minimum of affine functions of $g$. Hence, the finite game preserves the convex--concave structure, as claimed.
\end{proof}

It bears emphasis that the shape-preservation property hinges on the special mechanical payoff structure, not merely on the fact that the game entries are affine. In this payoff, the load slope is determined only by the displacement mixture chosen by the maximizing player, while the Dirichlet slope is determined only by the reaction mixture chosen by the minimizing player. This separability is what makes the finite value convex in the load variable and concave in the Dirichlet datum.

It follows that the finite saddle minmax is not merely a generic finite-dimensional approximation: it preserves the defining shape constraints of the mechanical value map. This property renders interpretation of optimal mixed strategies as mechanical readouts meaningful, as shown next.

\begin{proposition}[Mechanical readout of the saddle minmax]\label{prop:saddle-minmax-sense}
Let $(p,q)$ be a saddle strategy at $z$. Then,
\begin{equation}\label{eq:saddle-minmax-active-sensitivities}
    \bar u_p:=\sum_{i=1}^I p_i u_i,
    \quad
    \bar R_q:=\sum_{j=1}^J q_j R_j
\end{equation}
are, respectively, a load subgradient and a Dirichlet supergradient of $\mathcal G_{IJ}$ at $z$, i.e.,
\begin{equation}
    \bar u_p\in\partial_f\mathcal G_{IJ}(z),
    \quad
    \bar R_q\in\partial_g\mathcal G_{IJ}(z).
\end{equation}
At every point where the finite saddle value is differentiable with respect to the data,
\begin{equation}\label{eq:saddle-minmax-gradient}
    D\mathcal G_{IJ}(z)[\delta z]
    =\langle\delta f,\bar u_p\rangle_{X^*,X}+\bar R_q[\delta g].
\end{equation}
\end{proposition}

\begin{proof}
The proof is the same supporting-plane argument as in Proposition~\ref{prop:direct-sensitivity}, now applied to the finite matrix game. Holding the optimal load-side strategy $p$ fixed gives the load subgradient inequality. Holding the optimal reaction-side strategy $q$ fixed in the reversed minimax formulation gives the Dirichlet supergradient inequality. At differentiability points the supporting functional is unique.
\end{proof}

Proposition~\ref{prop:saddle-minmax-sense} shows that the finite game retains not only the shape of the exact value functional, but also its mechanical sensitivity interpretation. The next question is, therefore, how to choose atoms so that these finite mechanical readouts approximate the supporting displacement--reaction pairs of the exact value map on a prescribed data class.

\subsubsection{Uniform convergence from displacement density and automatic reaction norming}
\label{subsec:atom-selection}

We first enunciate the notation conventions and standing assumptions implied in the convergence statement and elsewhere henceforth. Let $K\subset Z$ be compact and define
\begin{equation} \label{nuGYUQ}
    U_K:=\{u_z:z\in K\}\subset X .
\end{equation}
For every $N$, we let
\begin{equation}
    U_N=\{u_1^N,\ldots,u_{I_N}^N\}\subset X,
    \quad
    A_N:=\operatorname{co}U_N ,
\end{equation}
and assume that the convex hulls $A_N$ are uniformly bounded in $X$, i.e., there is $M<\infty$ such that
\begin{equation} \label{ClTt0f}
    \sup_N\sup_{u\in A_N}\|u\|_X\le M,
    \quad
    \sup_{z\in K}\|u_z\|_X\le M .
\end{equation}
We note that the second uniform bound follows automatically from the assumptions of Theorem~\ref{thm:uniform-convergence}, since, by Lemma 1, the mapping $z \mapsto u_z$ is strongly continuous, whence, for compact $K$ the set $U_K$ is compact in $X$, hence bounded. Let $L_{K,M}$ be a Lipschitz constant, uniform in $z=(f,g)\in K$, for the mapping $u\mapsto \langle f,u\rangle_{X^*,X}-\mathcal E(u)$ on the bounded set
\begin{equation}
\begin{split}
    &
    \{u\in X \,:\, \|u\|_X\le M\}
    \, \cup \\ &
    \{u+\Lift(g-T_Du) \,:\, z=(f,g)\in K,\ \|u\|_X\le M\} .
\end{split}
\end{equation}
Such constant exists by (A8), the compactness of $K$, the boundedness of $T_D
D$ and $\mathcal E$, and the uniform bound (\ref{ClTt0f}). Let $\lambda>0$ be such that
\begin{equation} \label{10Osu4}
    \lambda\ge L_{K,M}\|\Lift\|_{\mathcal L(G_D,X)} .
\end{equation}
Finally, we define the \emph{residual set} as
\begin{equation}
    S_N
    :=
    \{g-T_Du \,:\, z=(f,g)\in K,\ u\in A_N\}\subset G_D .
\end{equation}

\begin{lemma}[Automatic finite reaction norming]
\label{lem:automatic-finite-reaction-norming}
Let $\varepsilon_N\downarrow 0$. Then, for every $N$ there exists a finite set $B_N=\{B_1^N,\ldots,B_{J_N}^N\}\subset G_D^*$ such that
\begin{subequations} \label{TsTGw9}
\begin{align}
    &   \label{rYiunj}
    \|B_j^N\|_{G_D^*}\le 1,
    \quad 1\le j\le J_N,
    \\ &    \label{ayBtJr}
    \sup_{1\le j\le J_N}B_j^N[\xi]
    \ge
    \|\xi\|_{G_D}-\varepsilon_N,
    \quad
    \xi\in S_N .
\end{align}
\end{subequations}
\end{lemma}

\begin{proof}
Since $U_N$ is finite, $A_N=\operatorname{co}U_N$ is compact in $X$. Since $K$ is compact and the map $(z,u)\mapsto g-T_Du$ is continuous from $K\times A_N$ to $G_D$, the residual set $S_N$ is compact in $G_D$. 

Set $\delta_N:=\varepsilon_N/2$. Choose a finite $\delta_N$-net $\xi_1^N,\ldots,\xi_{J_N}^N\in S_N$ such that
\begin{equation}
    S_N\subset
    \bigcup_{j=1}^{J_N}
    \{\xi\in G_D:\|\xi-\xi_j^N\|_{G_D}\le\delta_N\}.
\end{equation}
If $\xi_j^N=0$, set $B_j^N=0$. For every nonzero $\xi_j^N$, Hahn--Banach gives $B_j^N\in G_D^*$ such that
\begin{equation}
    \|B_j^N\|_{G_D^*}=1,
    \quad
    B_j^N[\xi_j^N]=\|\xi_j^N\|_{G_D}.
\end{equation}
whence (\ref{rYiunj}) follows. 

Given $\xi\in S_N$, choose $j$ such that $\|\xi-\xi_j^N\|_{G_D}\le\delta_N$. Then,
\begin{equation}
\begin{split}
    &
    B_j^N[\xi]
    =
    B_j^N[\xi_j^N]+B_j^N[\xi-\xi_j^N]
    \ge
    \|\xi_j^N\|_{G_D}
    -
    \|\xi-\xi_j^N\|_{G_D}
    \\ & \ge
    \|\xi\|_{G_D}
    -
    2\|\xi-\xi_j^N\|_{G_D}
    \ge
    \|\xi\|_{G_D}-\varepsilon_N ,
\end{split}
\end{equation}
whence (\ref{ayBtJr}) follows. 
\end{proof}

We now combine the displacement approximation property with the finite dual norming of the residual trace set. The result shows that, once the equilibrium displacements over a compact data class are approximated in the displacement space, the reaction atoms may be chosen so as to enforce the Dirichlet constraint uniformly, yielding convergence of the resulting finite saddle minmax value maps. 

With $B_N$ as in Lemma~\ref{lem:automatic-finite-reaction-norming} and $\lambda$ as in (\ref{10Osu4}), we choose the reaction atoms 
\begin{equation} \label{HQEiY6}
    R_j^N:=-\lambda B_j^N,
    \quad 1\le j\le J_N,
\end{equation}
and set 
\begin{equation}
    C_N:=\operatorname{co}\{R_1^N,\ldots,R_{J_N}^N\}\subset G_D^* .
\end{equation}
We finally set 
\begin{equation}
    \mathcal G_N:=\mathcal G_{I_NJ_N}, 
\end{equation}
with $\mathcal G_{IJ}$ defined as in (\ref{eq:saddle-minmax-finite}) from displacement atoms $\{u_1^N,\ldots,u_{I_N}^N\}$ and reaction atoms $\{R_1^N,\ldots,R_{J_N}^N\}$. 

The following lemma supplies a useful equivalent form of $\mathcal G_N$.

\begin{lemma} \label{v82cve}
For every $z=(f,g)\in Z$,
\begin{equation}
    \mathcal G_N(z)
    =
    \max_{p\in\Delta_{I_N}}
    \min_{R\in C_N}
    \sum_{i=1}^{I_N}
    p_i
    \Bigl(
    \langle f,u_i^N\rangle_{X^*,X}
    +
    R[g-T_Du_i^N]
    -
    E(u_i^N)
    \Bigr).
\end{equation}
\end{lemma}

\begin{proof}
For fixed $z\in Z$ and $p\in\Delta_{I_N}$, define
\begin{equation}
    \Psi_{z,p}(R)
    :=
    \sum_{i=1}^{I_N}
    p_i
    \Bigl(
    \langle f,u_i^N\rangle_{X^*,X}
    +
    R[g-T_Du_i^N]
    -
    E(u_i^N)
    \Bigr),
    \quad R\in C_N .
\end{equation}
The map $R\mapsto \Psi_{z,p}(R)$ is continuous and affine. Since $C_N$ is the convex hull of finitely many points, it is a compact convex polytope. Hence, by the extreme-point theorem for affine functions on a compact polytope, the minimum of $\Psi_{z,p}$ over $C_N$ is attained at an extreme point of $C_N$. The extreme points of $C_N$ are contained in $\{R_1^N,\ldots,R_{J_N}^N\}$. Therefore,
\begin{equation}
\begin{split}
    &
    \min_{R\in C_N}\Psi_{z,p}(R)
    =
    \min_{1\le j\le J_N}\Psi_{z,p}(R_j^N)
    = \\ &
    \min_{q\in\Delta_{J_N}}
    \sum_{i=1}^{I_N}\sum_{j=1}^{J_N}
    p_iq_j
    \Bigl(
    \langle f,u_i^N\rangle_{X^*,X}
    +
    R_j^N[g-T_Du_i^N]
    -
    E(u_i^N)
    \Bigr).
\end{split}
\end{equation}
Taking the maximum over $p\in\Delta_{I_N}$ gives exactly the finite saddle value $\mathcal G_{I_NJ_N}$ defined in (\ref{eq:saddle-minmax-finite}), as advertised. 
\end{proof}

\begin{theorem}[Uniform convergence from displacement density]
\label{thm:uniform-convergence}
Assume \emph{(A1)--(A9)}. Let $K \subset Z$ be compact. Suppose $A_N$ uniformly bounded in $X$ and  
\begin{equation}
    \alpha_N
    :=
    \sup_{z\in K}\inf_{1\le i\le I_N}
    \|u_i^N-u_z\|_X
    =
    \operatorname{dist}_X(U_K,U_N)
    \to 0 .
\end{equation}
Choose reaction atoms as in Lemma~\ref{lem:automatic-finite-reaction-norming} and (\ref{HQEiY6}). Then,
\begin{equation} \label{omcGEx}
    \sup_{z\in K}\bigl(\mathcal G_N(z)-\mathcal G(z)\bigr)
    \le
    \lambda\varepsilon_N 
    \to 0,
\end{equation}
and
\begin{equation} \label{qVcsos}
    \sup_{z\in K}\bigl(\mathcal G(z)-\mathcal G_N(z)\bigr)
    \le
    \left(
    L_{K,M}+\lambda\|T_D\|_{\mathcal L(X,G_D)}
    \right)\alpha_N .
\end{equation}
In particular, $G_N(z) \to G(z)$ uniformly on $K$.
\end{theorem}

\begin{proof}
Fix $z=(f,g)\in K$. We first prove the upper bound. Let
$p\in\Delta_{I_N}$ and set
\begin{equation}
    \bar u_p:=\sum_{i=1}^{I_N}p_i u_i^N\in A_N .
\end{equation}
Since $u\mapsto \langle f,u\rangle_{X^*,X}-\mathcal E(u)$ is concave,
\begin{equation}
    \sum_{i=1}^{I_N}
    p_i
    \left(
    \langle f,u_i^N\rangle_{X^*,X}
    -\mathcal E(u_i^N)
    \right)
    \le
    \langle f,\bar u_p\rangle_{X^*,X}
    -\mathcal E(\bar u_p).
\end{equation}
In addition,
\begin{equation}
    \sum_{i=1}^{I_N}p_iR[g-T_Du_i^N]
    =
    R[g-T_D\bar u_p].
\end{equation}
By Lemma~\ref{lem:automatic-finite-reaction-norming},
\begin{equation}
\begin{split}
    \min_{R\in C_N}R[g-T_D\bar u_p]
    & =
    -\lambda
    \sup_{1\le j\le J_N}B_j^N[g-T_D\bar u_p]
    \\ & \le
    -\lambda\|g-T_D\bar u_p\|_{G_D}
    +
    \lambda\varepsilon_N .
\end{split}
\end{equation}
Therefore,
\begin{equation}
\begin{split}
    &
    \min_{R\in C_N}
    \sum_{i=1}^{I_N}p_i
    \left(
    \langle f,u_i^N\rangle_{X^*,X}
    +R[g-T_Du_i^N]
    -\mathcal E(u_i^N)
    \right)
    \le \\ & \quad
    \langle f,\bar u_p\rangle_{X^*,X}
    -\mathcal E(\bar u_p)
    -
    \lambda\|g-T_D\bar u_p\|_{G_D}
    +
    \lambda\varepsilon_N .
\end{split}
\end{equation}
Define
\begin{equation}
    \widehat u_p:=\bar u_p+\Lift(g-T_D\bar u_p).
\end{equation}
Then, $T_D\widehat u_p=g$, so $\widehat u_p\in X_g$, and
\begin{equation}
\begin{split}
    &
    \langle f,\bar u_p\rangle_{X^*,X}
    -\mathcal E(\bar u_p)
    \le \\ &
    \langle f,\widehat u_p\rangle_{X^*,X}
    -\mathcal E(\widehat u_p)
    +
    L_{K,M}\|\Lift(g-T_D\bar u_p)\|_X
    \le \\ &
    \langle f,\widehat u_p\rangle_{X^*,X}
    -\mathcal E(\widehat u_p)
    +
    L_{K,M}\|\Lift\|_{\mathcal L(G_D,X)}
    \|g-T_D\bar u_p\|_{G_D}.
\end{split}
\end{equation}
Furthermore, by (\ref{10Osu4}) ,
\begin{equation}
    \langle f,\bar u_p\rangle_{X^*,X}
    -\mathcal E(\bar u_p)
    -
    \lambda\|g-T_D\bar u_p\|_{G_D}
    \le
    \langle f,\widehat u_p\rangle_{X^*,X}
    -\mathcal E(\widehat u_p)
    \le
    \mathcal G(z).
\end{equation}
Hence, for every $p\in\Delta_{I_N}$,
\begin{equation}
    \min_{R\in C_N}
    \sum_{i=1}^{I_N}p_i
    \left(
    \langle f,u_i^N\rangle_{X^*,X}
    +R[g-T_Du_i^N]
    -\mathcal E(u_i^N)
    \right)
    \le
    \mathcal G(z)+\lambda\varepsilon_N .
\end{equation}
Taking the maximum over $p$ and appealing to Lemma~\ref{v82cve} gives
\begin{equation}
    \mathcal G_N(z)\le \mathcal G(z)+\lambda\varepsilon_N ,
\end{equation}
whence (\ref{omcGEx}) follows. 

For the lower bound, choose $i=i(z,N)$ such that
\begin{equation}
    \|u_i^N-u_z\|_X\le\alpha_N .
\end{equation}
Using the pure strategy $p=e_i$ in the definition of $\mathcal G_N$, we obtain
\begin{equation}
    \mathcal G_N(z)
    \ge
    \min_{R\in C_N}
    \left(
    \langle f,u_i^N\rangle_{X^*,X}
    +R[g-T_Du_i^N]
    -\mathcal E(u_i^N)
    \right).
\end{equation}
Since $T_Du_z=g$, it follows that
\begin{equation}
    g-T_Du_i^N=T_D(u_z-u_i^N).
\end{equation}
Moreover, every $R\in C_N$ satisfies
\begin{equation}
    \|R\|_{G_D^*}\le\lambda .
\end{equation}
Hence,
\begin{equation}
    \min_{R\in C_N}R[g-T_Du_i^N]
    \ge
    -\lambda\|T_D\|_{\mathcal L(X,G_D)}
    \|u_i^N-u_z\|_X ,
\end{equation}
and the Lipschitz bound gives
\begin{equation}
    \langle f,u_i^N\rangle_{X^*,X}
    -\mathcal E(u_i^N)
    \ge
    \langle f,u_z\rangle_{X^*,X}
    -\mathcal E(u_z)
    -
    L_{K,M}\|u_i^N-u_z\|_X .
\end{equation}
Since $u_z$ is the equilibrium displacement for $z$, we have
\begin{equation}
        \langle f,u_z\rangle_{X^*,X}
        -
        \mathcal E(u_z)
        =
        \mathcal G(z).
\end{equation}
Combining these estimates yields
\begin{equation}
        \mathcal G_N(z)
        \ge
        \mathcal G(z)
        -
        \left(
        L_{K,M}
        +
        \lambda\|T_D\|_{\mathcal L(X,G_D)}
        \right)\alpha_N ,
\end{equation}
and taking the supremum over $z\in K$ proves (\ref{qVcsos}). 

The two one-sided estimates imply uniform convergence on $K$.
\end{proof}

We note that Theorem~\ref{thm:uniform-convergence} allows the reaction atoms to be selected from the full dual space $G_D^*$. For each $N$, the displacement atoms determine the compact residual set $S_N$. Lemma~\ref{lem:automatic-finite-reaction-norming} then selects finitely many dual functionals that norm $S_N$ up to $\varepsilon_N$. Thus, the only operative approximation assumption in the theorem is the displacement density condition $\operatorname{dist}_X(U_K,U_N)\to 0$ . The reaction atoms are not prescribed in advance; they are added afterward to enforce the Dirichlet constraint on the residuals produced by $A_N$. 

It bears emphasis that the automatic reaction atoms are generally nonconstructive. They are obtained by applying Hahn--Banach to residuals in a finite net of $S_N$. Consequently, they may be distributional elements of $G_D^*=W^{-1+1/p,q}(\Gamma_D;\mathbb R^d)$ and need not admit a simple representation as $L^q$-surface tractions or as normal traces of smooth stress fields. 

Lemma~\ref{lem:automatic-finite-reaction-norming} ensures that suitable reaction atoms always exist in $G_D^*$. However, in actual calculations, the reaction atoms may be restricted to a structured reaction dictionary. For instance, for traction-form atoms the norming functionals are of the form
\begin{equation}
    B_j^N[\xi]
    =
    \int_{\Gamma_D}\psi_j^N\cdot\xi\,ds ,
\end{equation}
whereas for volumetric atoms the norming functionals are of the form
\begin{equation}
    B_j^N[\xi]
    =
    \int_\Omega b_j^N\cdot \Lift\xi\,dx
    -
    \int_\Omega \sigma_j^N:e(\Lift\xi)\,dx .
\end{equation}
Then, the chosen dictionary must be shown to approximate the full-dual norming functionals on the residual sets $S_N$.

\begin{example}[Saddle minmax convergence for the elementary truss example]\label{ex:truss-minmax-convergence}
{\rm For the truss of Example~\ref{oE9UI9}, set $k:=EA/L$. The mixed saddle Lagrangian is
\begin{equation}
    \mathcal L((u,v,R);z)
    =2fu+R(g-v)-2ku^2-\frac{k}{2}(u+v)^2,
\end{equation}
and the exact sensitivities are
\begin{equation}\label{eq:truss-example2-sensitivities}
    u_z=\frac{2f}{5k}-\frac{g}{5},
    \qquad
    R_z=-k(u_z+g)=-\frac{2f+4kg}{5}.
\end{equation}
Substitution gives
\begin{equation}\label{eq:truss-example2-value}
    \mathcal G(z)=\frac{2f^2}{5k}-\frac{2kg^2}{5}-\frac{2fg}{5}.
\end{equation}

For the normalized square $K=[-1,1]^2$ and $k=1$, take a rectangular grid of displacement atoms $(u_i,v_i)$ covering the exact ranges of $(u_z,g)$, and take two reaction atoms $\pm R_*$ with $R_*=1.25$, so that their convex hull contains all exact reactions over $K$. The finite game is then
\begin{equation}\label{eq:truss-example2-computable-game}
    \mathcal G_N(z)
    =
    \min_{R\in[-R_*,R_*]}
    \max_{(u_i,v_i)}
    \{2fu_i+R(g-v_i)-\mathcal E(u_i,v_i)\}.
\end{equation}
The resulting surfaces converge to the exact convex--concave value function, Fig.~\ref{fig:Truss3}, with the observed second-order atom-refinement rate shown in Fig.~\ref{fig:truss-example2-errors}.

\begin{figure}[h!]
\centering
\includegraphics[width=\textwidth]{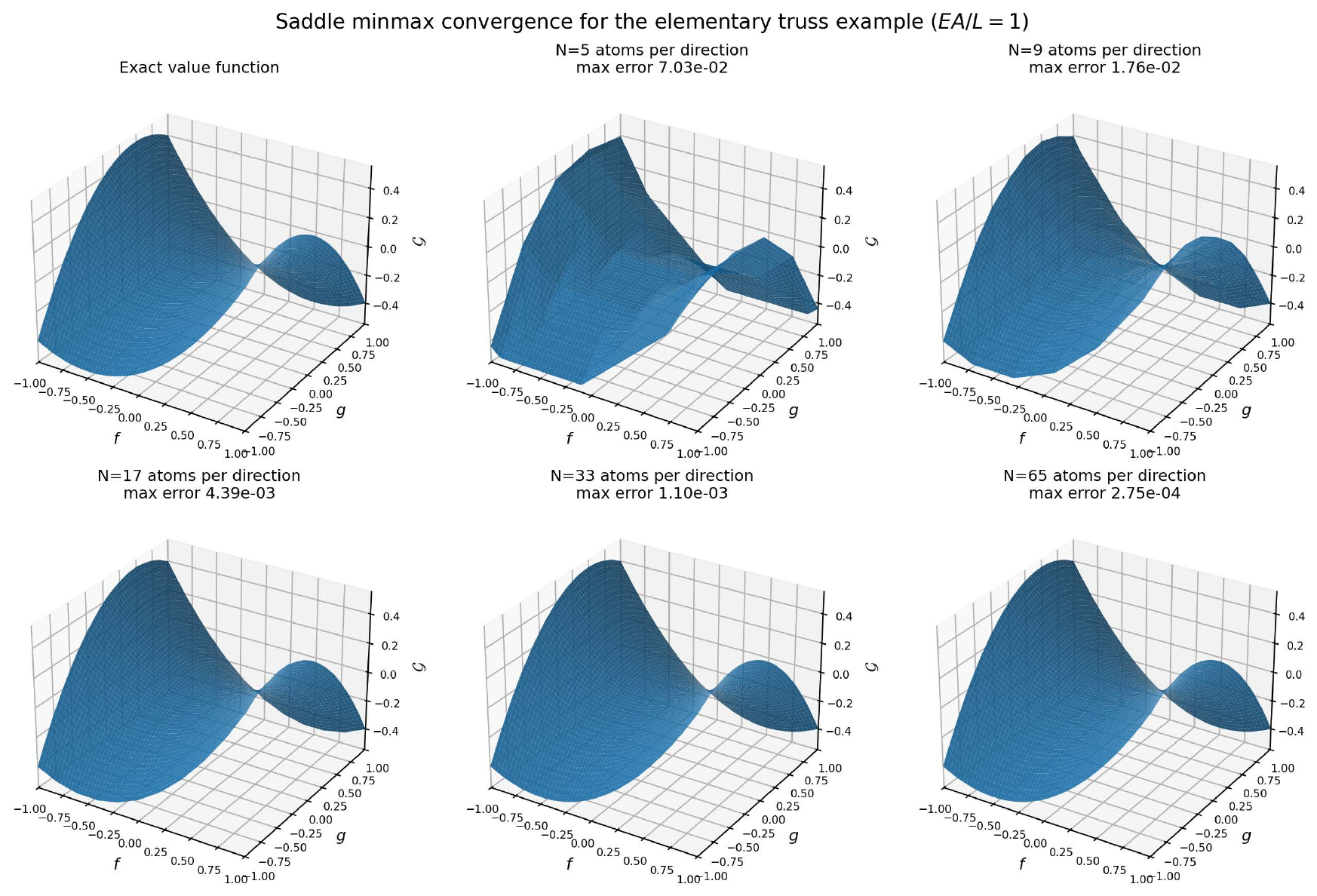}
\caption{Saddle minmax convergence for the elementary truss example. The first panel shows the exact value function \eqref{eq:truss-example2-value}; the remaining panels show finite saddle approximations \eqref{eq:truss-example2-computable-game} for increasingly refined displacement-atom grids.}
\label{fig:Truss3}
\end{figure}

\begin{figure}[h!]
\centering
\includegraphics[width=.75\textwidth]{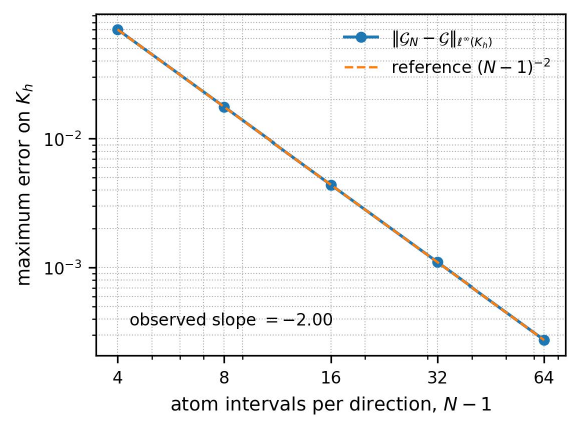}
\caption{Log--log convergence of the saddle minmax approximation for the truss example. The dashed guide is proportional to $(N-1)^{-2}$; the least-squares observed slope is $-2.00$.}
\label{fig:truss-example2-errors}
\end{figure}
} \hfill$\square$
\end{example}

\subsubsection{Weak convergence of mechanical readouts}

Theorem~\ref{thm:uniform-convergence} gives uniform convergence of the saddle minmax values. However, applications often concern other \emph{quantities of interest} obtained by applying bounded linear functionals to the mechanical readouts of $\mathcal G$, namely to its load and Dirichlet sensitivities. Such quantities of interest may represent averaged displacements, pointwise or localized response measurements when continuous, reaction resultants, or other experimentally accessible observables and design variables. The following corollary records the weak convergence statement needed to pass such quantities of interest to the limit.

\begin{corollary}[Weak convergence of mechanical readouts]
\label{cor:weak-convergence-readouts}
Assume the hypotheses of Theorem~\ref{thm:uniform-convergence} on every compact subset of $Z$, so that $\mathcal G_N\to\mathcal G$ locally uniformly on $Z$. Let $z_N\to z$ in $Z$. Let $(p_N,q_N)$ be a saddle strategy for $\mathcal G_N$ at $z_N$, and define
\begin{equation}
    u_N:=\sum_i p_{N,i}u_i^N,
    \quad
    R_N:=\sum_j q_{N,j}R_j^N .
\end{equation}
Assume, as in the construction of Theorem~\ref{thm:uniform-convergence}, that the displacement convex hulls are uniformly bounded in $X$ and that the reaction atoms are uniformly bounded in $G_D^*$. Then, up to subsequences,
\begin{equation}
    u_N\rightharpoonup u \in\partial_f\mathcal G(z) \;\, \text{in }X,
    \quad
    R_N\rightharpoonup R \in\partial_g\mathcal G(z) \;\, \text{in }G_D^*.
\end{equation}
If $\mathcal G$ is differentiable at $z$, then
\begin{equation}
    u=D_f\mathcal G(z),
    \quad
    R=D_g\mathcal G(z),
\end{equation}
and the entire sequence converges weakly to this pair.
\end{corollary}

\begin{proof}
The readouts satisfy $u_N\in A_N$ and $R_N\in \operatorname{co}\{R_1^N,\ldots,R_{J_N}^N\}$. By the assumed uniform boundedness and reflexivity of $X$ and $G_D^*$, subsequential weak limits exist. Proposition~\ref{prop:saddle-minmax-sense} gives $u_N\in\partial_f\mathcal G_N(z_N)$ and $R_N\in\partial_g\mathcal G_N(z_N)$.

Fix $\tilde f\in X^*$. The compact set
\begin{equation}
    K_{\tilde f}:=\{z\}\cup\{z_N:N\ge 1\}\cup\{(\tilde f,g)\}\cup\{(\tilde f,g_N):N\ge 1\}
\end{equation}
is compact in $Z$. The load-supporting inequality for $\mathcal G_N$ gives
\begin{equation}
    \mathcal G_N(\tilde f,g_N)
    \ge
    \mathcal G_N(f_N,g_N)+\langle \tilde f-f_N,u_N\rangle_{X^*,X}.
\end{equation}
Passing to the limit along a weakly convergent subsequence, using local uniform convergence on $K_{\tilde f}$, $z_N\to z$, and weak convergence of $u_N$, yields
\begin{equation}
    \mathcal G(\tilde f,g)
    \ge
    \mathcal G(f,g)+\langle \tilde f-f,u\rangle_{X^*,X}.
\end{equation}
Since $\tilde f$ is arbitrary, $u\in\partial_f\mathcal G(z)$.

The same argument in the Dirichlet variable gives, for every $\tilde g\in G_D$,
\begin{equation}
    \mathcal G(f,\tilde g)
    \le
    \mathcal G(f,g)+R[\tilde g-g],
\end{equation}
and hence $R\in\partial_g\mathcal G(z)$. At differentiability points the supporting pair is unique, so every weakly convergent subsequence has the same limit; boundedness then gives weak convergence of the entire sequence.
\end{proof}

The corollary shows that bounded mechanical readouts of the finite saddle games are stable under weak convergence and identify the corresponding supporting slopes of the exact value map. Consequently, any quantity of interest given by a bounded linear functional on $X$ or on $G_D^*$ converges along the same subsequences. At differentiability points of $\mathcal G$, the limiting readout is unique, and the observables converge to the quantities computed from $D_f\mathcal G(z)$ and $D_g\mathcal G(z)$ themselves.

\section{Fourier atoms and cell-center quadrature implementation}

This section turns the preceding abstract atom construction into computable Fourier dictionaries and quadrature rules. Specifically, displacement and reaction atoms are generated using finite Fourier expansions. The attendant integrals are then evaluated using numerical quadrature.

\subsection{Fourier displacement atoms and boundary reaction atoms}

\begin{figure}[h!]
\centering
\includegraphics[width=0.8\textwidth]{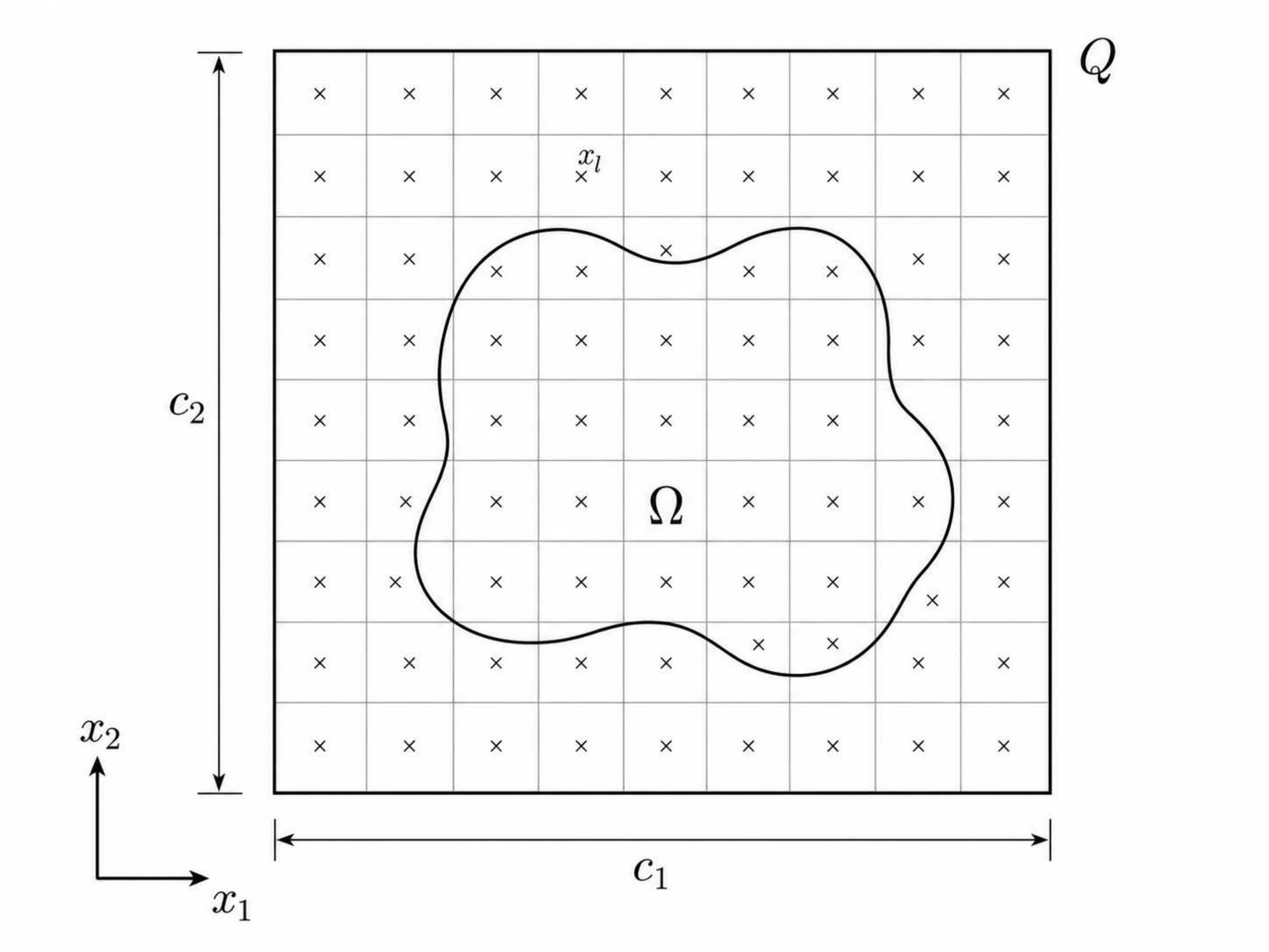}
\caption{Schema for the generation of manufactured displacement and reaction atoms using finite Fourier expansions and numerical quadrature.} \label{fig:Schema}
\end{figure}

A convenient source of manufactured displacement atoms is a finite Fourier expansion on a background cube, Fig.~\ref{fig:Schema}. Let
\begin{equation} \label{IFLb2D}
    Q:=\prod_{\alpha=1}^d(0,c_\alpha),
    \quad \Omega\Subset Q,
\end{equation}
and introduce the modes
\begin{equation}
    e_l(x):=|Q|^{-1/2}
    \exp\left(2\pi i\sum_{\alpha=1}^d l_\alpha\frac{x_\alpha}{c_\alpha}\right),
    \quad l\in\mathbb Z^d.
\end{equation}
Real-valued atoms are obtained from the associated sine--cosine basis or by imposing the conjugacy condition $a_{-l}=\overline a_l$ on coefficients. A typical manufactured displacement atom is, then,
\begin{equation}\label{eq:fourier-displacement-atom}
    u_i(x)=\sum_{l\in I_L}a_{i,l}e_l(x),
    \quad
    I_L:=\{l\in\mathbb Z^d:|l|_\infty\leq L\} .
\end{equation}
No projection onto a fixed boundary class is needed, because the trace $T_Du_i$ is part of the manufactured datum.

A convenient, purely volumetric, way of generating reaction atoms is to start from
stress-field atoms in the bulk, e.g., by a finite Fourier expansion of the form
\begin{equation}\label{eq:fourier-stress-atom}
    \sigma_j(x)=\sum_{m\in I_M}s_{j,m}e_m(x),
    \quad
    I_M:=\{m\in\mathbb Z^d:|m|_\infty\leq M\},
\end{equation}
with resultant body forces
\begin{equation}
    b_j=-\operatorname{div}\sigma_j .
\end{equation}
Define
\begin{equation}\label{eq:stress-generated-reaction-atom}
    R_j[\eta]
    :=
    \int_\Omega b_j\cdot \Lift\eta\,dx
    -
    \int_\Omega \sigma_j:e(\Lift\eta)\,dx,
    \quad
    \eta\in G_D .
\end{equation}
Then, $R_j\in G_D^*$ by H\"older's inequality and boundedness of the lifting $\Lift:G_D\to X$. Thus, stress atoms in $\Omega$ generate admissible reaction atoms without requiring explicit boundary quadrature. For the elementary saddle payoff, this strategy gives the computable volume formula
\begin{equation}\label{eq:stress-generated-reaction-payoff}
    R_j[g-T_Du_i]
    =
    \int_\Omega b_j\cdot \Lift(g-T_Du_i)\,dx
    -
    \int_\Omega \sigma_j : e\bigl(\Lift(g-T_Du_i)\bigr)\,dx .
\end{equation}
The corresponding coupling coefficient is
\begin{equation}\label{eq:stress-generated-coupling}
    C_{ij}
    =
    -R_j[T_Du_i]
    =
    -\int_\Omega b_j\cdot \Lift(T_Du_i)\,dx
    +
    \int_\Omega \sigma_j : e\bigl(\Lift(T_Du_i)\bigr)\,dx .
\end{equation}

If, in addition, the lifting is chosen to have zero trace on $\Gamma_N$, or the stress atom has homogeneous Neumann traction in the corresponding weak sense, then \eqref{eq:stress-generated-reaction-atom} is precisely the weak volume representation of the Dirichlet reaction induced by $\sigma_j$, with the standing sign convention for the reaction variable. In the present manufactured-data construction, however, the definition \eqref{eq:stress-generated-reaction-atom} is sufficient: it defines an element of $G_D^*$ entirely through integrals over $\Omega$.

\begin{proposition}[Density of Fourier displacement atoms]
\label{prop:fourier-displacement-density}
Assume \emph{(A1)--(A9)}. Let $K \subset Z$ be compact. Then, there exists a sequence of finite collections $U_L=\{u^L_1,\ldots,u^L_{I_L}\}\subset X$ of Fourier displacement atoms such that
\begin{equation}
    \alpha_L
    :=
    \sup_{z\in K}\inf_{1\le i\le I_L}
    \|u^L_i-u_z\|_X
    \to  0,
    \quad\text{as }L\to\infty .
\end{equation}
Consequently, the Fourier displacement atoms satisfy the displacement-density assumption of Theorem~\ref{thm:uniform-convergence}.
\end{proposition}

\begin{proof}
By Lemma~\ref{lem:strong-continuity-compactness}, the equilibrium map $z\mapsto u_z$ is strongly continuous from $Z$ to $X$. Since $K$ is compact, $U_K$ is compact in $X$, hence bounded.

Next, we note that restrictions to $\Omega$ of finite Fourier sums on $Q$ are dense in $X$. This uses the strict immersion $\overline\Omega\subset Q$. Let $u\in W^{1,p}(\Omega;\mathbb R^d)$ and $\varepsilon>0$. By the Sobolev extension theorem for Lipschitz domains, extend $u$ to a field on a slightly larger Lipschitz set contained in $Q$. Choose a smooth cutoff equal to one on $\overline\Omega$ and compactly supported in $Q$. After multiplication by this cutoff and standard mollification, obtain $\varphi\in C_c^\infty(Q;\mathbb R^d)$ whose restriction to $\Omega$ is within $\varepsilon/2$ of $u$ in $W^{1,p}(\Omega;\mathbb R^d)$. The zero extension of $\varphi$ to the periodic torus associated with $Q$ belongs to $W^{1,p}_{\rm per}(Q;\mathbb R^d)$, and its periodic Fourier partial sums, or equivalently their Ces\`aro means, converge to it strongly in $W^{1,p}(Q;\mathbb R^d)$. Restricting to $\Omega$ gives, for every $u\in X$ and every $\varepsilon>0$, a finite Fourier sum
\begin{equation}
    v(x)=\sum_{|l|_\infty\le L} a_l e_l(x)\big|_{\Omega},
    \quad
    \|v-u\|_X<\varepsilon .
\end{equation}

Let $\varepsilon>0$. Since $U_K$ is compact in $X$, there are finitely many fields $u^1,\ldots,u^m\in U_K$ such that
\begin{equation}
    U_K\subset \bigcup_{r=1}^m
    \{u\in X:\|u-u^r\|_X<\varepsilon/2\}.
\end{equation}
For every $r=1,\ldots,m$, choose a finite Fourier sum $v^r$ such that $\|v^r-u^r\|_X < \varepsilon/2$. Let $L_\varepsilon$ be the largest Fourier order needed among these finitely many approximants and define $U_{L_\varepsilon}:=\{v^1,\ldots,v^m\}$. Then, for every $z\in K$, there is an index $r$ such that $\|u_z-u^r\|_X<\varepsilon/2$, and, therefore,
\begin{equation}
    \inf_{v\in U_{L_\varepsilon}}\|u_z-v\|_X
    \le
    \|u_z-v^r\|_X
    \le
    \|u_z-u^r\|_X+\|u^r-v^r\|_X
    <\varepsilon .
\end{equation}
Taking the supremum over $z\in K$ further gives
\begin{equation}
    \sup_{z\in K}\inf_{v\in U_{L_\varepsilon}}
    \|u_z-v\|_X<\varepsilon .
\end{equation}
Since $\varepsilon>0$ is arbitrary, we may choose a sequence $\varepsilon_L\downarrow0$ and corresponding finite Fourier dictionaries $U_L$ with
\begin{equation}
    \sup_{z\in K}\inf_{u_i^L\in U_L}
    \|u_i^L-u_z\|_X\le \varepsilon_L\to0 ,
\end{equation}
which is the displacement-density condition.
\end{proof}

The abstract norming step in Lemma~\ref{lem:automatic-finite-reaction-norming} guarantees the existence of finitely many reaction functionals in the full dual space $G_D^*$, but it does not by itself identify a computable reaction dictionary. For the Fourier construction, this gap is closed by showing that the volumetric reaction atoms generated by Fourier stress fields are sufficiently rich to reproduce the same finite norming property on each residual net. This is the content of the next lemma.

We shall use the following trace-density and trace-realization facts; see, for instance, Adams--Fournier~\cite{AdamsFournier2003} and McLean~\cite{McLean2000}. The mixed-boundary trace setting is understood in the admissible-patch sense used in~\cite{Groeger1989, HallerDintelmannJonssonKneesRehberg2016, AuscherBadrHallerDintelmannRehberg2015}. Smooth Dirichlet traction functionals are dense in $G_D^*=W^{-1+1/p,q}(\Gamma_D;\mathbb R^d)$. Moreover, every smooth Dirichlet traction may be represented, up to arbitrary accuracy in $G_D^*$, by the weak normal trace of a smooth symmetric stress field on $\Omega$ with no unwanted contribution on $\Gamma_N$; equivalently, either the lifting $\Lift\eta$ is chosen with zero trace on $\Gamma_N$, or the stress field is chosen with homogeneous weak normal trace on $\Gamma_N$. Finally, restrictions to $\Omega$ of finite Fourier symmetric stress fields on the background cube are dense in $W^{1,q}(\Omega;\mathbb S^d)$ for the compatible stress class under consideration.

\begin{lemma}[Fourier reaction atoms norm finite residual nets]
\label{lem:fourier-reactions-norm-residual-net}
Assume the trace-density and Fourier stress-density facts stated above. Let $\Xi_N = \{\xi_1^N, \ldots, \xi_{m_N}^N\}\subset S_N$ be a finite residual net. Then, the Fourier reaction atoms norm the finite residual net $\Xi_N$ up to a prescribed error $\rho_N>0$.
\end{lemma}

\begin{proof}
Fix $a\in\{1,\ldots,m_N\}$. If $\xi_a^N=0$, the zero reaction atom suffices. Otherwise, by Hahn--Banach and the density of smooth Dirichlet tractions in $G_D^*$, there exists a smooth traction $t_a^N$ whose induced functional
\begin{equation}
    T_a^N[\eta] := \int_{\Gamma_D} t_a^N \cdot \eta \, ds
\end{equation}
satisfies
\begin{equation}
    T_a^N[\xi_a^N] \ge \|\xi_a^N\|_{G_D} - \frac{\rho_N}{3},
    \qquad
    \|T_a^N\|_{G_D^*} \le 1+\gamma_a .
\end{equation}
where $\gamma_a>0$ is chosen small enough below. Choose a smooth symmetric stress field $\tau_a^N$ on $\overline\Omega$ whose weak normal trace realizes the reaction traction ${t}_a^N=-\tau_a^N\nu$ on $\Gamma_D$ and gives no contribution on $\Gamma_N$ in the sense specified above, i.e.,
\begin{equation}
        \int_\Omega (-\operatorname{div}\tau_a^N)\cdot \Lift\eta\,dx
        -
        \int_\Omega \tau_a^N:e(\Lift\eta)\,dx
        =
        {T}_a^N[\eta],
        \quad \eta\in G_D .
\end{equation}
Indeed, for symmetric $\tau_a^N$, integration by parts gives the left-hand side as $-\int_{\partial\Omega}(\tau_a^N\nu)\cdot\Lift\eta\,ds$, so the Dirichlet reaction traction is $t_a^N=-\tau_a^N\nu$ in the sign convention of \eqref{eq:Dirichlet-reaction}.
Approximate $\tau_a^N$ in $W^{1,q}(\Omega;\mathbb S^d)$ by a finite Fourier stress field $\sigma_a^N$ restricted from the background cube. Since the map
\begin{equation}
    \sigma\mapsto
    \left[
    \eta\mapsto
    \int_\Omega (-\operatorname{div}\sigma)\cdot \Lift\eta\,dx
    -
    \int_\Omega \sigma:e(\Lift\eta)\,dx
    \right]
\end{equation}
is continuous from $W^{1,q}(\Omega;\mathbb S^d)$ into $G_D^*$, the approximation may be chosen so that the induced Fourier reaction functional
\begin{equation}
    \widetilde B_a^N[\eta]
    :=
    \int_\Omega (-\operatorname{div}\sigma_a^N)\cdot \Lift\eta\,dx
    -
    \int_\Omega \sigma_a^N:e(\Lift\eta)\,dx
\end{equation}
satisfies
\begin{equation}
    \widetilde B_a^N[\xi_a^N]
    \ge
    \|\xi_a^N\|_{G_D}-\frac{2\rho_N}{3},
    \quad
    \|\widetilde B_a^N\|_{G_D^*}\le 1+\gamma_a .
\end{equation}
Finally, define
\begin{equation}
    B_a^N:=\frac{\widetilde B_a^N}
    {\max\{1,\|\widetilde B_a^N\|_{G_D^*}\}} .
\end{equation}
By taking $\gamma_a$ sufficiently small, the normalization changes the value on $\xi_a^N$ by at most $\rho_N/3$. Hence,
\begin{equation}
    \|B_a^N\|_{G_D^*}\le 1,
    \quad
    B_a^N[\xi_a^N]\ge \|\xi_a^N\|_{G_D}-\rho_N .
\end{equation}
Repeating this construction for every nonzero point of the finite set $\Xi_N$, and adding the zero atom if needed, gives the required finite collection of Fourier reaction atoms.
\end{proof}

The preceding density argument shows that Fourier reaction atoms are rich enough, in principle, to norm any prescribed finite residual net, provided that the chosen structured reaction dictionary is dense enough to norm the compact residual sets appearing in Lemma~\ref{lem:automatic-finite-reaction-norming}. For use in the convergence theorem, it is useful to isolate the exact finite condition that the reaction dictionary must satisfy. The next proposition records this criterion independently of the particular construction of the atoms.

\begin{proposition}[A finite norming criterion for Fourier reaction atoms]
\label{prop:fourier-reaction-norming}
Let $\delta_N>0$, $\rho_N>0$, and $\varepsilon_N > 0$ such that $\rho_N+2\delta_N\le \varepsilon_N$, and let $\Xi_N = \{\xi_1^N, \ldots, \xi_{m_N}^N \}\subset S_N$ be a finite $\delta_N$-net for $S_N$, i.e.,
\begin{equation}
    S_N\subset
    \bigcup_{a=1}^{m_N}
    \{\xi\in G_D:\|\xi-\xi_a^N\|_{G_D}\le \delta_N\}.
\end{equation}
Then, there exists a finite collection of Fourier stress atoms $\sigma_1^N,\ldots,\sigma_{J_N}^N$ that satisfy the finite reaction-norming condition (\ref{TsTGw9}) with error $\varepsilon_N$. 
\end{proposition}

\begin{proof}
By Lemma~\ref{lem:fourier-reactions-norm-residual-net}, there exists a finite family of Fourier reaction atoms $B_1^N,\ldots,B_{J_N}^N$ of the form (\ref{eq:stress-generated-reaction-atom}) such that
\begin{subequations}
\begin{align}
    &
    \|B_j^N\|_{G_D^*}\le 1,
    \quad 1\le j\le J_N,
    \\ &
    \max_{1\le j\le J_N} B_j^N[\xi_a^N]
    \ge
    \|\xi_a^N\|_{G_D}-\rho_N,
    \quad a=1,\ldots,m_N .
\end{align}
\end{subequations}
Let $\xi\in S_N$. Since $\Xi_N$ is a $\delta_N$-net for $S_N$, there exists $a\in\{1,\ldots,m_N\}$ such that
\begin{equation}
    \|\xi-\xi_a^N\|_{G_D}\le \delta_N .
\end{equation}
Choose $j$ such that
\begin{equation}
    B_j^N[\xi_a^N]\ge \|\xi_a^N\|_{G_D}-\rho_N .
\end{equation}
Using $\|B_j^N\|_{G_D^*}\le 1$, we obtain
\begin{equation}
\begin{split}
    &
    B_j^N[\xi]
    =
    B_j^N[\xi_a^N]+B_j^N[\xi-\xi_a^N]
    \ge
    \|\xi_a^N\|_{G_D}-\rho_N-\|\xi-\xi_a^N\|_{G_D}
    \ge \\ &
    \|\xi\|_{G_D}
    -2\|\xi-\xi_a^N\|_{G_D}
    -\rho_N
    \ge
    \|\xi\|_{G_D}-(\rho_N+2\delta_N).
\end{split}
\end{equation}
Taking the supremum over $j$ proves
\begin{equation}
    \sup_{1\le j\le J_N} B_j^N[\xi]
    \ge
    \|\xi\|_{G_D}-(\rho_N+2\delta_N),
    \quad \xi\in S_N .
\end{equation}
If $\rho_N+2\delta_N\le\varepsilon_N$, this is precisely the norming
condition (\ref{TsTGw9}). 
\end{proof}

We note that, scaling a Fourier stress atom $\sigma_j^N$ by $-\lambda$ scales $b_j^N=-\operatorname{div}\sigma_j^N$ and the corresponding reaction functional by the same factor, so the scaled atoms remain of the form (\ref{eq:stress-generated-reaction-atom}).

\begin{corollary}[Convergence of the Fourier manufactured saddle scheme]
\label{cor:fourier-saddle-convergence}
Assume (A1)--(A9). Assume, in addition, that the Fourier-generated reaction atoms norm the
compact residual sets as in Lemma\ref{lem:fourier-reactions-norm-residual-net}. Then, for every compact set $K \subset Z$, Fourier displacement atoms and Fourier-generated reaction atoms can be chosen such that the manufactured saddle scheme converges uniformly to $\mathcal G$ on $K$.
\end{corollary}

\begin{proof}
Proposition~\ref{prop:fourier-displacement-density} supplies the displacement-density hypothesis of Theorem~\ref{thm:uniform-convergence}. Proposition ~\ref{prop:fourier-reaction-norming} supplies the finite reaction-norming condition (\ref{TsTGw9}), with $\varepsilon_N\to0$, using reaction atoms of the volumetric form (\ref{eq:stress-generated-reaction-atom}). Therefore, Theorem ~\ref{thm:uniform-convergence} gives the one-sided bounds
\begin{equation}
    \sup_{z\in K}\bigl(\mathcal G_N(z)-\mathcal G(z)\bigr)\le \lambda\varepsilon_N
    \to  0
\end{equation}
and
\begin{equation}
    \sup_{z\in K}\bigl(\mathcal G(z)-\mathcal G_N(z)\bigr)
    \le
    \bigl(L_{K,M}+\lambda\|T_D\|_{\mathcal L(X,G_D)}\bigr)
    \alpha_N
    \to  0 .
\end{equation}
Combining the two estimates proves uniform convergence of $\mathcal G_N$ to $\mathcal G$ on $K$.
\end{proof}

\subsection{Cell-center quadrature}
We now make explicit the numerical quadrature used in Fig.~\ref{fig:Schema}. Let $Q$ be as in (\ref{IFLb2D}) and let $\mathcal Q_h$ be the uniform Cartesian partition of $Q$ into rectangular cells $K$, with cell centers $x_K$ and cell volumes $|K|$. For an integrand $\varphi$ defined on $\Omega$, extended by zero to $Q$, we use the midpoint-cell quadrature
\begin{equation}\label{eq:quadrature-rule}
    Q_h[\varphi]
    := \sum_{K\in\mathcal Q_h} |K|\,
    \mathbf 1_{\Omega}(x_K)\,\varphi(x_K).
\end{equation}
Thus, the only sampling points are the centers of the cells in the background cube $Q$, exactly as indicated in Fig.~\ref{fig:Schema}. Boundary cells are classified by whether their center lies in $\Omega$. Since $\Omega$ is Lipschitz, $|\partial\Omega|=0$, and therefore this center-point rule is consistent for every bounded Riemann-integrable integrand on $\Omega$, equivalently for every bounded function whose zero extension to $Q$ is continuous except on a set of Lebesgue measure zero.

For a manufactured displacement atom $u_i$ generated by the finite Fourier expansion (\ref{eq:fourier-displacement-atom}), define
\begin{equation}\label{eq:quadrature-atom-energy}
    \mathcal E_h(u_i)
    := Q_h\left[\,W(\cdot,e(u_i))\,\right].
\end{equation}
For a manufactured stress atom $\sigma_j$ generated by (\ref{eq:fourier-stress-atom}), set $b_j=-\operatorname{div}\sigma_j$. The reaction functional (\ref{eq:stress-generated-reaction-atom}) is then approximated by the purely volumetric formula
\begin{equation}\label{eq:quadrature-reaction-functional}
    R_{j,h}[\eta]
    :=
    Q_h\left[
    b_j\cdot \Lift\eta
    -
    \sigma_j : e(\Lift\eta)
    \right],
    \quad \eta\in G_D .
\end{equation}
In particular,
\begin{equation}\label{eq:quadrature-coupling-coefficients}
    C_{ij,h}
    :=
    -R_{j,h}[T_Du_i]
    =
    -Q_h\left[
    b_j\cdot \Lift(T_Du_i)
    -
    \sigma_j : e(\Lift(T_Du_i))
    \right].
\end{equation}
No independent boundary quadrature is required: the Dirichlet-reaction pairing is represented through the volume identity defining $R_j$. The quadrature payoff is
\begin{equation}\label{eq:quadrature-payoff}
    \Phi_{ij,h}(z)
    :=
    \langle f,u_i\rangle_h
    +
    R_{j,h}[g-T_Du_i]
    -
    \mathcal E_h(u_i).
\end{equation}
When the external load contains a classical Neumann part, the quadrature symbol $\langle f,u_i\rangle_h$ is understood to include the chosen quadrature or ambient-dual representation of that load. Thus, if
\begin{equation}
    f(\xi)=\int_\Omega b\cdot \xi\,dx+\langle h,\operatorname{Tr}\xi\rangle_{\Gamma_N},
\end{equation}
then
\begin{equation}
    \langle f,u_i\rangle_h
    :=
    Q_h[b\cdot u_i]+
    \langle h,\operatorname{Tr}u_i\rangle_{\Gamma_N,h},
\end{equation}
where the second term is either a boundary quadrature on $\Gamma_N$ or an equivalent midpoint-consistent ambient-dual representation of the Neumann load. The statement that no independent boundary quadrature is required applies to the Dirichlet-reaction coupling when reaction atoms are represented volumetrically; it does not preclude the quadrature used to represent prescribed external Neumann loads.
The corresponding cell-center quadrature approximation of the saddle game
(\ref{eq:saddle-minmax-finite}) is
\begin{equation}\label{eq:quadrature-saddle-game}
    \mathcal G_{IJ,h}(z)
    :=
    \max_{p\in\Delta_I}\min_{q\in\Delta_J}
    \sum_{i=1}^I\sum_{j=1}^J p_iq_j\Phi_{ij,h}(z).
\end{equation}

We next record the elementary stability fact that the cell-center quadrature does not spoil the uniform convergence of the finite saddle approximations. The point is that a finite zero-sum matrix game is Lipschitz, with constant one, with respect to the entrywise maximum norm. 

\begin{lemma}[Uniform payoff consistency of the cell-center rule]
\label{9rJ4Qg}
Let $K\subset Z$ be compact and let $u_1,\ldots,u_I$ and $R_1,\ldots,R_J$ be fixed finite atom families generated by the Fourier constructions (\ref{eq:fourier-displacement-atom}) and (\ref{eq:stress-generated-reaction-atom}). Assume that the scalar integrands entering the three terms of the payoff are uniformly consistent for the cell-center rule on the compact data class $K$, namely,
\begin{subequations}
\begin{align}
\max_{1\le i\le I}| \mathcal E_h(u_i)-\mathcal E(u_i)| &\to 0,\\
\sup_{z=(f,g)\in K}\max_{1\le i\le I}|\langle f,u_i\rangle_h-\langle f,u_i\rangle_{X^*,X}| &\to 0,\\
\sup_{z=(f,g)\in K}\max_{\substack{1\le i\le I\\1\le j\le J}}|R_{j,h}[g-T_Du_i]-R_j[g-T_Du_i]| &\to 0,
\end{align}
\end{subequations}
as $h\downarrow0$. Then, the cell-center quadrature payoffs (\ref{eq:quadrature-payoff}) satisfy the uniform consistency condition
\begin{equation} \label{n0GU3T}
    \eta_h :=
    \sup_{z\in K}
    \max_{1\le i\le I,\;1\le j\le J}
    \bigl|\Phi_{ij,h}(z)-\Phi_{ij}(z)\bigr|
    \to 0 ,
    \quad\text{as } h\downarrow 0 .
\end{equation}
\end{lemma}

\begin{proof}
For every $z=(f,g)\in K$ and every $i,j$,
\begin{equation}
\begin{split}
    &
    |\Phi_{ij,h}(z)-\Phi_{ij}(z)|
    \le
    |\langle f,u_i\rangle_h-\langle f,u_i\rangle_{X^*,X}| 
    + \\ &
    |R_{j,h}[g-T_Du_i]-R_j[g-T_Du_i]|
    +
    |\mathcal E_h(u_i)-\mathcal E(u_i)| .
\end{split}
\end{equation}
Taking the supremum over $z\in K$ and the maximum over the finite atom families gives (\ref{n0GU3T}) from the three assumed uniform consistency conditions.
\end{proof}

It bears emphasis that the preceding quadrature result is conditional on uniform payoff consistency for the represented atom and data classes; it is not a pointwise quadrature claim for arbitrary Sobolev representatives. We also note that the uniform consistency assumptions in Lemma~\ref{9rJ4Qg} are quadrature assumptions on the particular finite payoff family, not additional well-posedness assumptions on the continuum problem. They hold, for example, when the loads in the compact data class are represented by a uniformly midpoint-consistent family of bounded Riemann-integrable volume densities on $Q$, when the Dirichlet data are lifted by a uniformly midpoint-consistent smooth or piecewise-smooth family, and when the displacement and stress atoms are finite Fourier sums.

\begin{proposition}[Convergence with respect to quadrature]
\label{prop:cell-center-quadrature-uniform}
Let $K\subset Z$ be compact. For fixed finite atom families $u_1,\ldots,u_I\in X$, $R_1,\ldots,R_J\in G_D^*$, let $\mathcal G_{IJ}$ be the exact finite saddle value defined by (\ref{eq:saddle-minmax-finite}), and let $\mathcal G_{IJ,h}$ be the cell-center quadrature value defined by (\ref{eq:quadrature-saddle-game}). Assume that the quadrature payoffs are uniformly consistent on $K$ as in (\ref{n0GU3T}). Then,
\begin{equation}
    \sup_{z\in K}
    \bigl|
        \mathcal G_{IJ,h}(z)-\mathcal G_{IJ}(z)
    \bigr|
    \le \eta_h
    \to  0 .
\end{equation}
\end{proposition}

\begin{proof}
Fix $z\in K$, and write
\begin{equation}
    A(z) := \bigl(\Phi_{ij}(z)\bigr)_{i,j},
    \quad
    A_h(z) := \bigl(\Phi_{ij,h}(z)\bigr)_{i,j}.
\end{equation}
The corresponding matrix-game values are
\begin{equation}
    \mathcal G_{IJ}(z)
    =
    \max_{p\in\Delta_I}\min_{q\in\Delta_J}p^T A(z)q,
    \quad
    \mathcal G_{IJ,h}(z)
    =
    \max_{p\in\Delta_I}\min_{q\in\Delta_J}p^T A_h(z)q .
\end{equation}
For every $p\in\Delta_I$ and $q\in\Delta_J$,
\begin{equation}
    \bigl|p^T(A_h(z)-A(z))q\bigr|
    \le
    \max_{i,j}
    \bigl|
        \Phi_{ij,h}(z)-\Phi_{ij}(z)
    \bigr|.
\end{equation}
Taking first the minimum over $q$, then the maximum over $p$, and
reversing the roles of $A_h(z)$ and $A(z)$, gives
\begin{equation}
    |\mathcal G_{IJ,h}(z)-\mathcal G_{IJ}(z)|
    \le
    \max_{i,j}
    \bigl|
        \Phi_{ij,h}(z)-\Phi_{ij}(z)
    \bigr|.
\end{equation}
Taking the supremum over $z\in K$ proves the first assertion.
\end{proof}

\begin{corollary}[Cell-center quadrature does not spoil uniform convergence]
Let $\mathcal G_N$ be any sequence of exact finite saddle approximations satisfying
\begin{equation}
    \sup_{z\in K}|\mathcal G_N(z)-\mathcal G(z)|\to  0 .
\end{equation}
If $h_N\downarrow0$ is chosen so that
\begin{equation}
    \eta_{N,h_N}
    :=
    \sup_{z\in K}
    \max_{1\le i\le I_N,\;1\le j\le J_N}
    \bigl|
        \Phi^N_{ij,h_N}(z)-\Phi^N_{ij}(z)
    \bigr|
    \to 0 ,
\end{equation}
then, the fully computable quadrature saddle values satisfy
\begin{equation}
    \sup_{z\in K}|\mathcal G_{N,h_N}(z)-\mathcal G(z)|\to 0 .
\end{equation}
Thus, under uniform payoff consistency, the cell-center quadrature does not spoil the uniform convergence of the manufactured saddle minmax scheme on compact data classes.
\end{corollary}

\begin{proof}
Apply the first estimate to the $N$-th atom family and use the triangle inequality:
\begin{equation}
\begin{split}
    &
    \sup_{z\in K}|\mathcal G_{N,h_N}(z)-\mathcal G(z)|
    \le \\ &
    \sup_{z\in K}|\mathcal G_{N,h_N}(z)-\mathcal G_N(z)|
    +
    \sup_{z\in K}|\mathcal G_N(z)-\mathcal G(z)|
    \le \\ &
    \eta_{N,h_N}
    +
    \sup_{z\in K}|\mathcal G_N(z)-\mathcal G(z)|.
\end{split}
\end{equation}
Both terms on the right-hand side tend to zero by assumption. This proves the fully quadrature-based uniform convergence.
\end{proof}

\subsection{Learning algorithm}\label{subsec:learning-algorithm}

\begin{algorithm}[h!]
\caption{Cell-center construction of a finite saddle minmax game}
\begin{algorithmic}[1]
\Require Number of displacement atoms $I$, number of reaction atoms $J$,
cell-center quadrature rule $Q_h$, trace lifting $\Lift:G_D\to X$, samplers for
displacement atoms and reaction atoms
\Ensure Cell-center saddle approximation $\mathcal G_{IJ,h}$
\State Generate displacement atoms $u_1,\ldots,u_I\in X$
\For{$i=1,\ldots,I$}
    \State Compute $\mathcal E_h(u_i)$ using \eqref{eq:quadrature-atom-energy}
\EndFor
\State Generate quadrature reaction functionals
$R_{1,h},\ldots,R_{J,h}$ using \eqref{eq:quadrature-reaction-functional}
or another chosen reaction-atom evaluator
\For{$i=1,\ldots,I$ and $j=1,\ldots,J$}
    \State Compute $C_{ij,h}$ using
    \eqref{eq:quadrature-coupling-coefficients}
    \State Assemble $\Phi_{ij,h}$ using \eqref{eq:quadrature-payoff}
\EndFor
\State Return the saddle value $\mathcal G_{IJ,h}$ defined by
\eqref{eq:quadrature-saddle-game}
\end{algorithmic}
\end{algorithm}

The computable learning procedure consists of selecting displacement atoms and reaction atoms, assembling the quadrature coefficients defined in \eqref{eq:quadrature-atom-energy}--\eqref{eq:quadrature-coupling-coefficients}, and evaluating the cell-center saddle game \eqref{eq:quadrature-saddle-game}. Thus, the algorithm does not introduce independent intercepts or coupling coefficients: the stored-energy terms are those of \eqref{eq:quadrature-atom-energy}, the reaction functionals are those of \eqref{eq:quadrature-reaction-functional}, and the coupling matrix is fixed by \eqref{eq:quadrature-coupling-coefficients}. The payoff entries are then the quadrature payoffs \eqref{eq:quadrature-payoff}.

This is the mixed-boundary analogue of selecting full-Neumann minmax slopes and maxout networks from manufactured equilibria in the full-Neumann endpoint. The additional ingredient is the reaction atom sampler, which supplies the concave Dirichlet-side slopes. Once a saddle strategy has been computed for \eqref{eq:quadrature-saddle-game}, the corresponding active displacement and reaction readouts are read from \eqref{eq:saddle-minmax-active-sensitivities}, with $R_j$ replaced by $R_{j,h}$.

It remains to specify how the finite-dimensional game \eqref{eq:quadrature-saddle-game} is evaluated. For each datum $z=(f,g)$, set
\begin{equation}
    A_h(z):=\big(A_{ij,h}(z)\big)_{i=1,\ldots,I;\,j=1,\ldots,J},
    \quad
    A_{ij,h}(z):=\Phi_{ij,h}(z).
\end{equation}
Then,
\begin{equation}
    \mathcal G_{IJ,h}(z)=\max_{p\in\Delta_I}\min_{q\in\Delta_J}
    p^T A_h(z) q .
\end{equation}
Thus, evaluation of \eqref{eq:quadrature-saddle-game} is exactly the solution of a finite zero-sum matrix game. Since both strategy sets are simplices, the value and an optimal pair of mixed strategies may be obtained from the primal--dual pair of linear programs
\begin{subequations}\label{eq:quadrature-game-lps}
\begin{align}
    \mathcal G_{IJ,h}(z)
    &=\max_{p,\vartheta}\ \vartheta ,
    &&\text{subject to:}\;\,
    \sum_{i=1}^I p_i A_{ij,h}(z)\ge \vartheta,
    \;\, j=1,\ldots,J, \nonumber\\
    &&& \sum_{i=1}^I p_i=1,\;\, p_i\ge0,
    \;\, i=1,\ldots,I, \label{eq:quadrature-game-primal-lp}
    \\
    \mathcal G_{IJ,h}(z)
    &=\min_{q,\vartheta}\ \vartheta ,
    &&\text{subject to:}\;\,
    \sum_{j=1}^J A_{ij,h}(z)q_j\le \vartheta,
    \;\, i=1,\ldots,I, \nonumber\\
    &&& \sum_{j=1}^J q_j=1,\;\, q_j\ge0,
    \;\, j=1,\ldots,J. \label{eq:quadrature-game-dual-lp}
\end{align}
\end{subequations}
The linear-programming formulation is the most direct method for small and moderate atom dictionaries, and standard simplex or interior-point methods may be used \cite{vonNeumannMorgenstern1944, Dantzig1963, Chvatal1983}. It also gives a useful certificate of accuracy. For arbitrary $p\in\Delta_I$ and $q\in\Delta_J$,
\begin{equation}
    \min_{1\le j\le J}(A_h(z)^T p)_j
    \le \mathcal G_{IJ,h}(z)
    \le \max_{1\le i\le I}(A_h(z)q)_i ,
\end{equation}
and hence the duality gap
\begin{equation}
    \operatorname{gap}(p,q;z)
    := \max_{1\le i\le I}(A_h(z)q)_i
    -  \min_{1\le j\le J}(A_h(z)^T p)_j
\end{equation}
provides a computable stopping criterion. At an exact saddle strategy $(p,q)$, the complementary slackness conditions are
\begin{subequations}
\begin{align}
    &
    p_i>0 \;\, \Rightarrow \;\, (A_h(z)q)_i=\mathcal G_{IJ,h}(z), \quad 1\le i\le I,
    \\ &
    q_j>0 \;\, \Rightarrow \;\, (A_h(z)^Tp)_j=\mathcal G_{IJ,h}(z), \quad 1\le j\le J,
\end{align}
\end{subequations}
together with the inequalities
\begin{equation}
    (A_h(z)q)_i\le \mathcal G_{IJ,h}(z)\le (A_h(z)^Tp)_j ,
    \quad
    1\le i\le I, \;\, 1\le j\le J .
\end{equation}
Consequently, once the active supports are known, the game may also be solved by support enumeration or active-set equalization, followed by these inequality checks \cite{vonStengel2002}.

For large dictionaries, all entries of $A_h(z)$ need not be formed explicitly, provided matrix-vector products can be evaluated. If
\begin{equation}
    \bar u_p:=\sum_{i=1}^I p_i u_i ,
    \quad
    \bar R_{q,h}:=\sum_{j=1}^J q_j R_{j,h},
\end{equation}
then,
\begin{subequations}
\begin{align}
    (A_h(z)q)_i
    &=
    \langle f,u_i\rangle_h-\mathcal E_h(u_i)
    +\bar R_{q,h}[g-T_Du_i],
    \;\,
    1\le i\le I
    \label{eq:matrix-game-row-product}
    \\
    (A_h(z)^Tp)_j
    &=
    \langle f,\bar u_p\rangle_h
    -\sum_{i=1}^I p_i\mathcal E_h(u_i)
    +R_{j,h}[g-T_D\bar u_p],
    \;\,
    1\le j\le J.
    \label{eq:matrix-game-column-product}
\end{align}
\end{subequations}
These formulae support first-order saddle solvers, for instance entropic mirror descent or multiplicative-weights iterations,
\begin{subequations}
\begin{align}
    p_i^{n+1}
    &=
    \frac{p_i^n\exp\{\eta_n(A_h(z)q^n)_i\}}
    {\sum_{k=1}^I p_k^n\exp\{\eta_n(A_h(z)q^n)_k\}},
    \quad
    1\le i\le I
    \\
    q_j^{n+1}
    &=
    \frac{q_j^n\exp\{-\eta_n(A_h(z)^Tp^n)_j\}}
    {\sum_{k=1}^J q_k^n\exp\{-\eta_n(A_h(z)^Tp^n)_k\}} 
    \quad
    1\le j\le J,
\end{align}
\end{subequations}
with the ergodic averages tested by the duality gap above. Such matrix-free methods are useful when the atom sets are large or when the same atom dictionary is queried for many data $z$; higher-order or extra gradient variants may be used when a sharper first-order saddle residual is desired \cite{FreundSchapire1999, Nemirovski2004}. In all cases the active mechanical readouts are then recovered from the optimal, or approximately optimal, strategies by \eqref{eq:saddle-minmax-active-sensitivities}, with $R_j$ replaced by $R_{j,h}$.

\section{Illustrative numerical example}
\label{sec:dam-example}

We conclude with a small-strain Hookean example that illustrates the finite saddle minmax data-to-solution architecture in a simple setting. The purpose is not to present a finely converged engineering analysis, but to show how the value-map construction can be queried after the displacement and reaction dictionaries have been assembled. In contrast with a conventional solver, the primary object is to \emph{learn} the scalar value functional; the displacement and base reaction are then recovered as subgradient and supergradient readouts with respect to the load and Dirichlet data.

\begin{figure}[h!]
  \centering
  \includegraphics[width=0.72\textwidth]{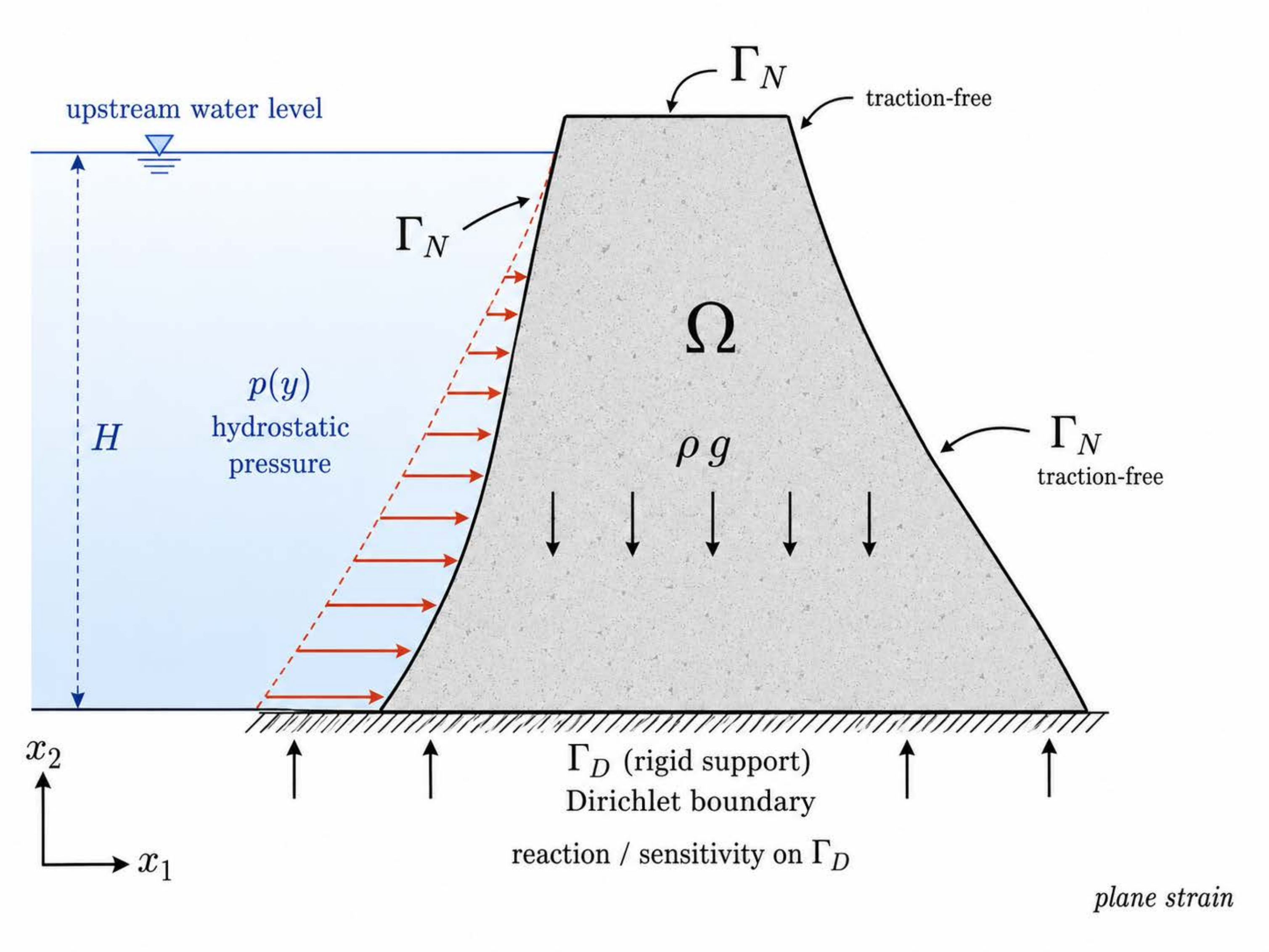}
  \caption{Plane-strain gravity-dam geometry used in the Hookean example. The base $\Gamma_D$ is clamped, the remaining boundary is $\Gamma_N$, hydrostatic pressure acts on the upstream face, and gravity acts in the interior.}
  \label{fig:dam-hookean-schematic}
\end{figure}

\begin{figure}[h!]
  \centering
  \includegraphics[width=0.96\textwidth]{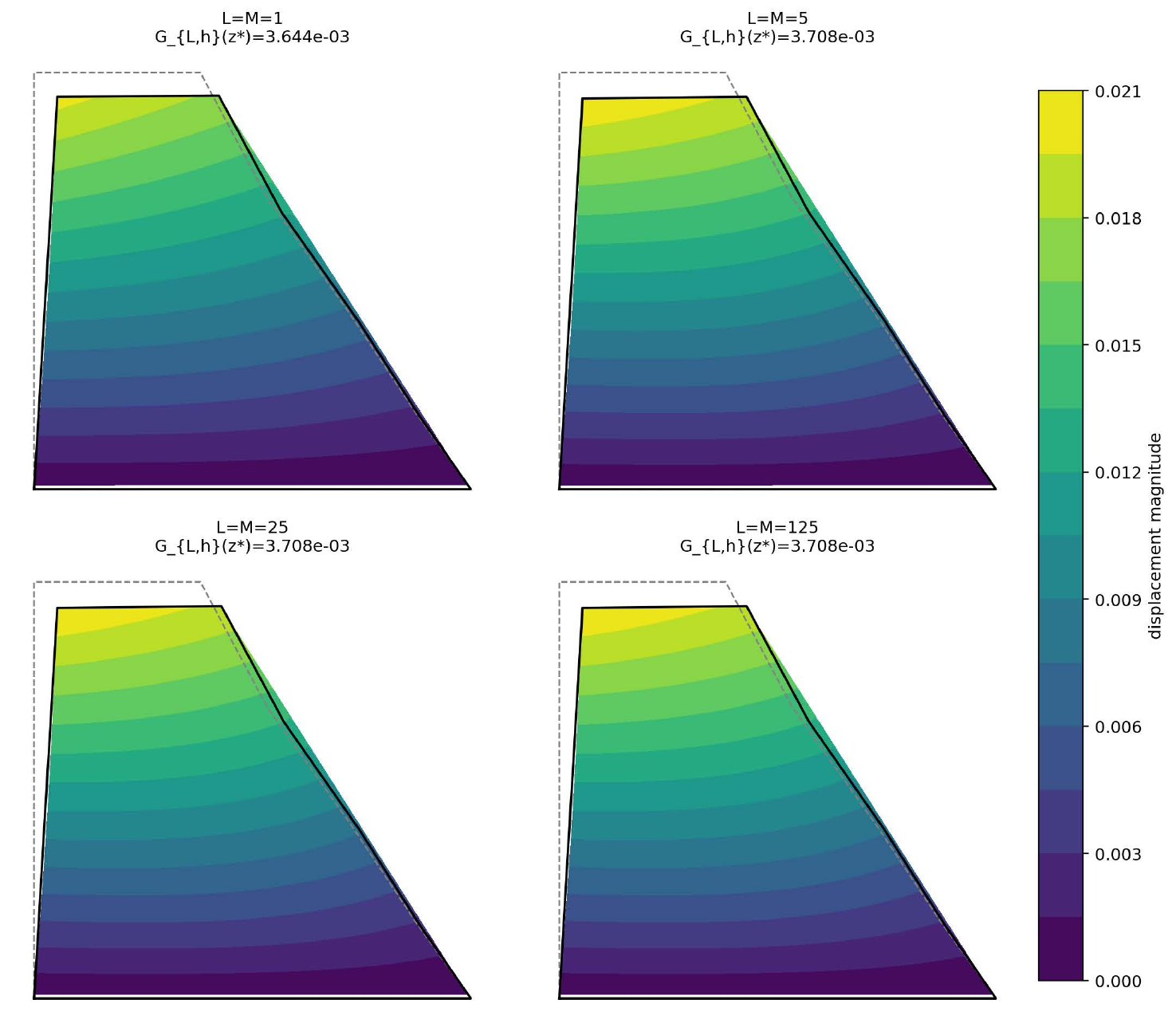}
  \caption{Queried displacement readouts for $L=M=1,5,25,125$. The contours show $|u_L|$ on the deformed configuration, with a common visual magnification. The dashed outline is the reference configuration.}
  \label{fig:dam-hookean-displacement-sequence}
\end{figure}

\begin{figure}[h!]
  \centering
  \includegraphics[width=0.96\textwidth]{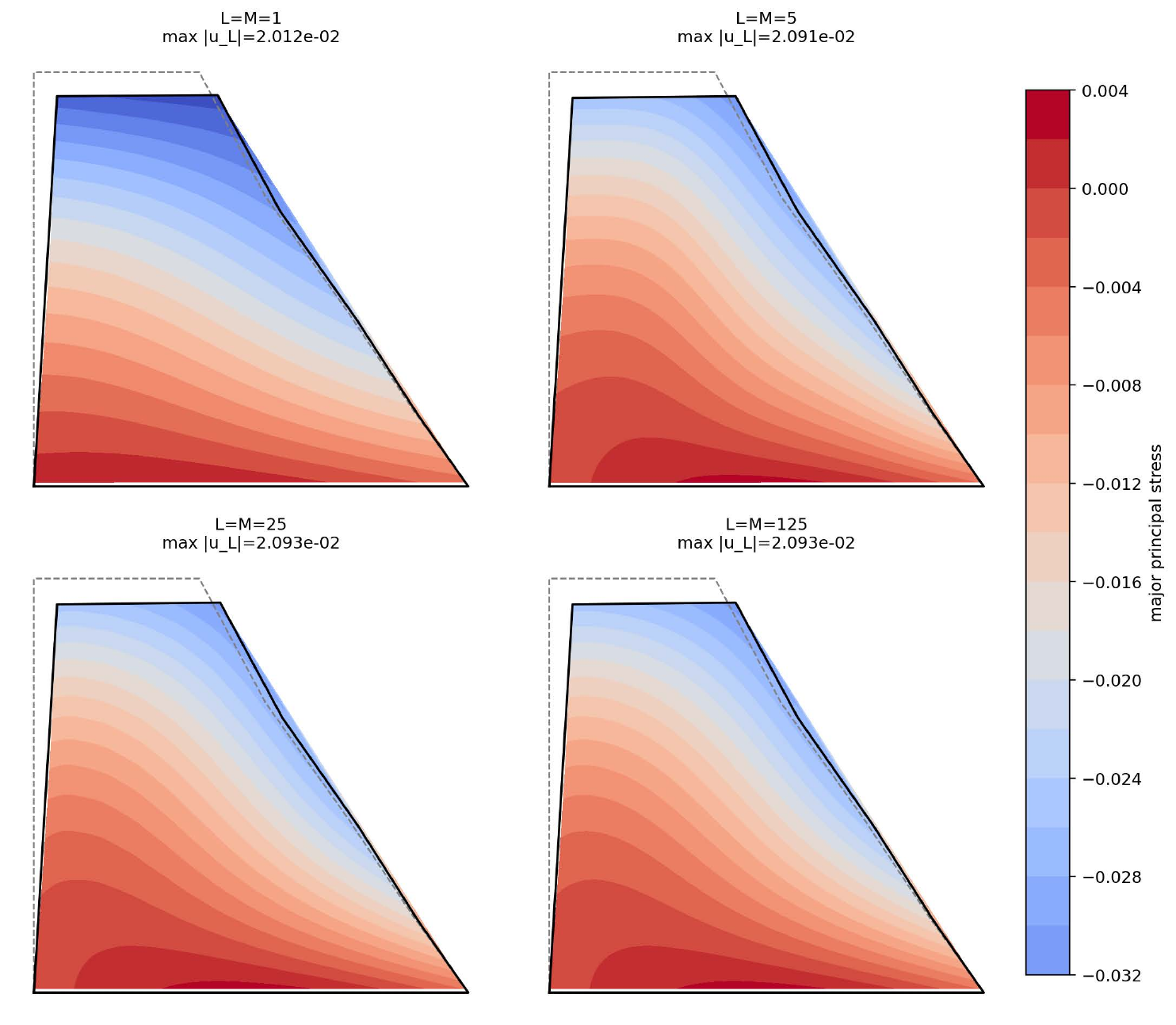}
  \caption{Major principal stress fields associated with the same queried sequence. The Hookean stress is obtained from the recovered displacement readout through \eqref{eq:dam-hookean-stress}. The stress field is smoother than in the subquadratic illustration because the constitutive law is linear and the displayed high-frequency content is filtered by the rapidly decaying manufactured dictionary.}
  \label{fig:dam-hookean-principal-stress-sequence}
\end{figure}

\begin{figure}[h!]
  \centering
  \includegraphics[width=0.96\textwidth]{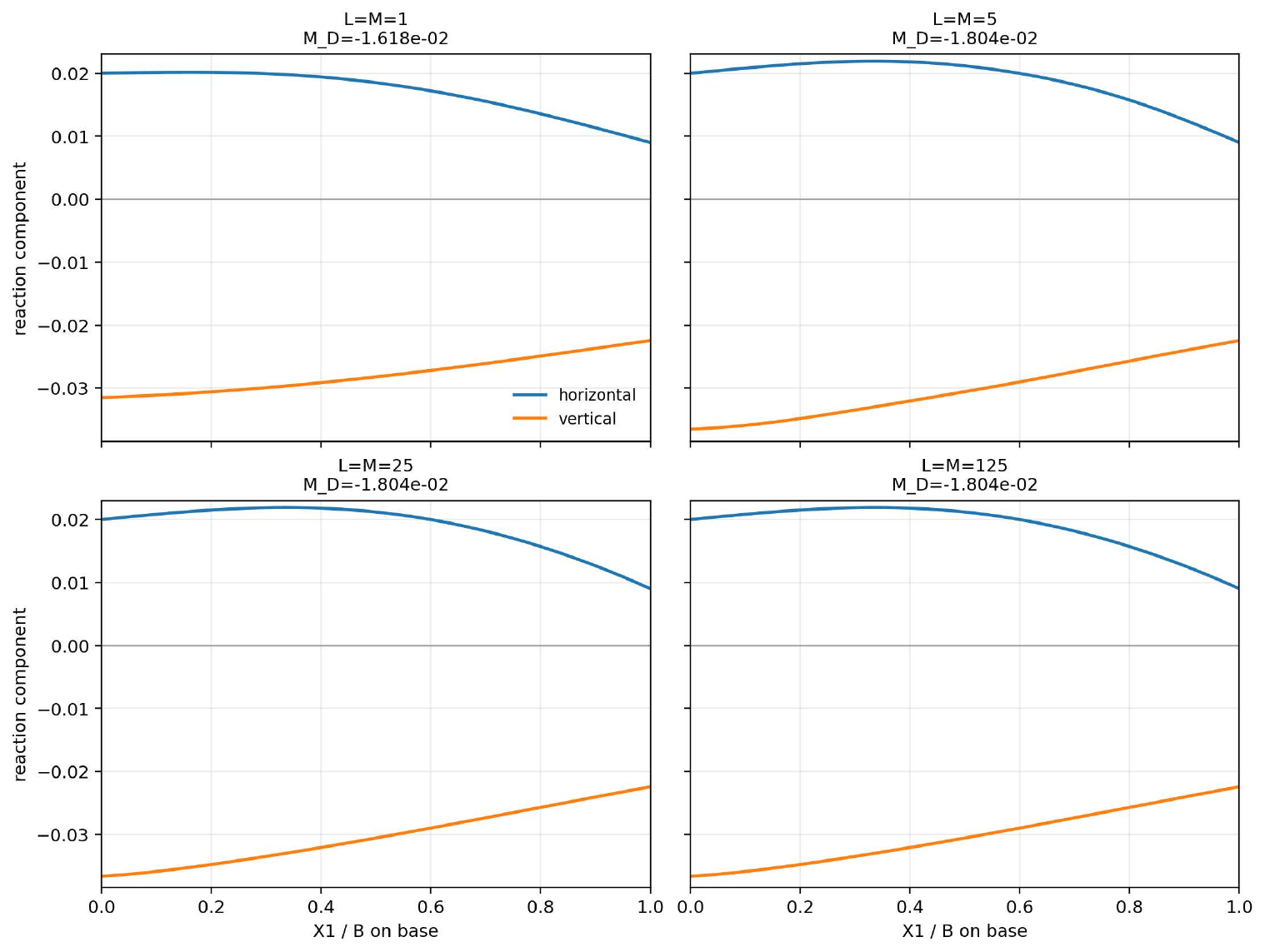}
  \caption{Recovered base reactions $R_L$ for $L=M=1,5,25,125$. The horizontal and vertical components are shown along $\Gamma_D$, together with the resultant overturning moment $M_D=\int_{\Gamma_D}(B-X_1)(R_L)_2\,ds$ about the downstream toe.}
  \label{fig:dam-hookean-base-reaction-sequence}
\end{figure}

The reference configuration is the polygonal cross section shown in Fig.~\ref{fig:dam-hookean-schematic}, with vertices: $(0,0)$, $(1.05,0)$, $(0.92,0.18)$, $(0.76,0.42)$, $(0.56,0.70)$, $(0.40,1.0)$, $(0,1.0)$, in counterclockwise order. The height is $H=1$ and the base width is $B=1.05$. The base $\Gamma_D=\{(X_1,0):0<X_1<B\}$ is clamped, while the remaining boundary is Neumann. The load class used for the value map is the compact two-parameter family
\begin{equation}
    K_{\rm dam}
    :=
    \{(\alpha f_h+\beta f_b,0):\ \alpha\in[0.8,1.4],\ \beta\in[0.25,0.45]\}
    \subset Z,
\end{equation}
where $f_h$ is the hydrostatic load on the upstream face and $f_b$ is the vertical body-force load. The query datum used in the figures is $z_*=(f_*,0)$, $f_*=f_h+f_b$, with traction and body force
\begin{equation}
    t_h(X_2)=p_0\Bigl(1-\frac{X_2}{H}\Bigr)e_1,
    \quad
    b=-\rho g\,e_2,
\end{equation}
where $p_0=1$ and $\rho g=0.30$.

The material is isotropic Hookean in the small-strain sense. Thus, for $A=e(u)$,
\begin{equation}\label{eq:dam-hookean-W}
    W(A)=\mu |\operatorname{dev}A|^2+\frac{\kappa}{2}(\operatorname{tr}A)^2,
\end{equation}
with shear modulus $\mu=1$ and bulk modulus $\kappa=2$. The corresponding Cauchy stress in the linearized model is
\begin{equation}\label{eq:dam-hookean-stress}
    \sigma(A)=2\mu\operatorname{dev}A+\kappa(\operatorname{tr}A)I.
\end{equation}
This choice is the Hilbert--Hookean case. In particular, the solution operator is affine in the data and the value functional is quadratic on the data space. The example therefore avoids the additional numerical complications caused by nonsmooth or subquadratic tangent behavior.

For each bandwidth $L=M$, we assemble a finite saddle approximation using base-compatible displacement atoms and base reaction atoms. The displacement dictionary contains low-order polynomial seed modes together with Fourier corrections of the form
\begin{equation} \label{ze4a0l}
  \phi_m^{(1)}(X)=\frac{X_2}{H}\sin \Bigl(\frac{m\pi X_1}{B}\Bigr)e_1,
  \quad
  \phi_m^{(2)}(X)=\frac{X_2}{H}\cos \Bigl(\frac{(m-\tfrac12)\pi X_1}{B}\Bigr)e_2,
\end{equation}
used in the two displacement components. In this example it is expedient to use a structured finite reaction dictionary on the Dirichlet boundary; it is a concrete implementation of the saddle architecture rather than the full automatic reaction-norming construction of Lemma~\ref{lem:automatic-finite-reaction-norming}. Specifically, the reaction side is represented componentwise on $\Gamma_D$ by the one-dimensional Fourier atoms
\begin{equation}
  \psi_m(X_1)=\cos \Bigl(\frac{m\pi X_1}{B}\Bigr).
\end{equation}
The manufactured displacement--reaction pairs determine the saddle payoffs, labels, and coupling coefficients. Once the finite saddle game has been assembled on $K_{\rm dam}$, the particular state associated with $z_*$ is obtained by evaluating the value map and reading off the active saddle subgradient and supergradient pair.

\begin{table}[h!]
  \centering
  \caption{Query values and queried mechanical readouts for the Hookean dam example. The finite saddle architecture is built once for each bandwidth pair and then evaluated at $z_*$.}
  \label{tab:dam-hookean-results}
  \small
  \begin{tabular}{c c c c c}
    \hline
    $L=M$ & dofs & $\mathcal G_{L,h}(z_*)$ & $\max_{\Omega}|u_L|$ & moment $M_D$\\
    \hline
    1 & 2 & 3.64383e-03 & 2.012e-02 & -1.61756e-02\\
    5 & 10 & 3.70847e-03 & 2.091e-02 & -1.80364e-02\\
    25 & 50 & 3.70833e-03 & 2.093e-02 & -1.80367e-02\\
    125 & 250 & 3.70833e-03 & 2.093e-02 & -1.80367e-02\\
    \hline
  \end{tabular}
\end{table}

The sequence in Figs.~\ref{fig:dam-hookean-displacement-sequence}--\ref{fig:dam-hookean-base-reaction-sequence} should be read as a sequence of queries of increasingly rich saddle architectures, not as four unrelated forward solves. The scalar value, displacement readout, and reaction readout are all returned by the same value-map representation. As the atom bandwidth is increased, the displacement field and the base reaction settle rapidly, while the stress field gains moderate detail near the upstream face and the support. This behavior is consistent with the compactness--density theorem: the Fourier manufactured dictionaries approximate the paired displacement--reaction support set, and the cell-center quadrature contributes only a vanishing payoff error. The Hookean setting makes this example especially transparent because the value map is quadratic and the solution operator is affine in the load data.

\section{Summary and conclusions}

We have developed a variational and approximation framework for the negative minimum-potential value functional of linearized elasticity with generalized loads and prescribed Dirichlet data. The key modeling distinction is between the physical external-load space ${L}_{\rm ext}$, which excludes independently prescribed tractions on $\Gamma_D$, and the larger ambient dual $X^*$ used for analysis and for augmented manufacturing. A bounded lifting of the Dirichlet trace converts the moving admissible class $X_g$ into a fixed zero-trace space, and the lift-load term $\langle f,\Lift g\rangle_{X^*,X}$ is retained so that the reduced problem has the same value as the original mixed-boundary problem.

We restrict attention to linearized kinematics for both analytic and structural reasons. The corresponding finite-kinematics boundary-value theory remains substantially more delicate, with several foundational questions still open \cite{Ball2002,BallMarsden1984}. The restriction to linearized kinematics has two well-known consequences that are used throughout. First, the mechanical variable is the displacement $u$, and the stored energy depends on $u$ only through the infinitesimal strain $e(u):=\operatorname{sym}Du=(Du+Du^T)/2$. Second, superposed rigid-body motions are infinitesimal rigid displacements $r(x)=a+Sx$, $a\in\mathbb R^d$, $S^T=-S$, for which $e(r)=0$. Requiring invariance under such motions is therefore equivalent to formulating the energy as a functional of $e(u)$ alone. In mixed problems with a nontrivial Dirichlet part, the boundary datum fixes these rigid motions. In full-Neumann problems they must be quotiented out or fixed by a gauge, and loads must be balanced against every rigid motion.

The analysis shows that load data and Dirichlet data play different structural roles. With $g$ fixed, the value functional is convex in $f$ because it is a supremum of affine load functionals. With $f$ fixed, it is concave in $g$ because it is the negative of the minimum of a convex functional over an affine trace class. Thus, the full mixed-boundary map is generally neither jointly convex nor jointly concave; it has a saddle geometry. This observation determines the architecture: load-only value maps admit a single convex maxout envelope, whereas full mixed-boundary data are naturally represented by a finite max--min game over displacement atoms and reaction atoms. The coupling matrix is not an arbitrary parameter matrix but the variational pairing $-R_j[T_Du_i]$.

The subgradient and supergradient formulae provide the mechanical interpretation of the saddle slopes. Variation with respect to the generalized load is represented by equilibrium displacements, while variation with respect to the prescribed Dirichlet displacement is represented by Dirichlet reactions in the sign convention used here. The equilibrium displacement set $U_K\subset X$ and the residual trace sets generated by its finite convex approximants are therefore the natural objects to control when selecting atoms. The compactness--density theorem of Section~6 makes this precise: convergence of displacement dictionaries to $U_K$, together with finite reaction norming and payoff consistency, implies convergence of the value functional on compact data classes, and the cell-center quadrature analysis shows that quadrature refinement does not spoil that convergence, provided the finite payoff family, including the chosen representation of prescribed external loads, is uniformly consistent. In quadratic Hookean settings the nonconstant part of the support map is linear, so optimal displacement--reaction feature spaces are characterized by singular-value, eigenvalue, entropy, or proper-orthogonal-decomposition constructions \cite{KolmogorovTikhomirov1961,Sirovich1987,BerkoozHolmesLumley1993}. The balanced full-Neumann compliant modes are recovered by deleting the reaction component.

Manufactured solutions simplify in the augmented generalized-load formulation. For every admissible manufactured displacement $u$ and every prescribed reaction $R\in G_D^*$, the datum
\begin{equation}
    f_w=D\mathcal E(u)+T_D^*R,
    \quad
    g_w=T_Du
\end{equation}
turns $u$ into the exact equilibrium displacement and $R$ into the exact Dirichlet reaction. The associated value label is
\begin{equation}
    \mathcal G(f_w,g_w)=D\mathcal E(u)[u]+R[g_w]-\mathcal E(u).
\end{equation}
The canonical choice $R=0$ recovers a standard manufactured load but does not sample nonzero Dirichlet-reaction slopes. Nonzero reactions are physical external-load samples only when $f_w\in{L}_{\rm ext}$; otherwise they should be understood as augmented samples designed to expose the concave Dirichlet side of the value map. This is the main difference from traditional MMS, V\&V, and UQ workflows: the manufactured fields are not used primarily to verify a forward solver, but to generate exact value-and-support data for a learner.

This distinction also clarifies the relation to immersed methods. Conventional immersed-boundary, immersed finite-element, cut finite-element, and mesh-free methods solve boundary-value problems on background grids or nonconforming approximation spaces. The present construction may use the same background grids, CAD descriptions, point clouds, quadrature rules, and Fourier or kernel features, but its target is a learned value functional whose derivatives deliver displacement and reaction information. In this sense the method is immersed in its representation but variational and learning-based in its objective.

The formulation is well suited to CAD-defined and mesh-free settings because it is expressed through domain data, volumetric energy densities and volumetric duality pairings rather than through a particular mesh. In particular, the formulation obviates the need for numerical quadrature rules over the boundary. Local correction terms, such as (\ref{ze4a0l}), and \emph{ad hoc} feature maps may be added from immersed sampling over the CAD volume and CAD-compatible basis functions fitted, e.g., to interfaces, sharp edges and corners, in order to enhance accuracy. Establishing sharp approximation rates and stability estimates for such CAD-induced features, especially in the presence of trimming, edges, corners, material interfaces, contact, and topological transitions, remains an important but challenging direction for future work.

\section*{Acknowledgements}

MO gratefully acknowledges the financial support of the {\sl Centre Internacional de M\`etodes Num\`erics a l'Enginyeria} (CIMNE) of the {\sl Universitat Politecnica de Catalunya} (UPC), Spain, through the {\sl UNESCO Chair in Numerical Methods in Engineering}.

\end{document}